\documentclass[11pt]
{amsart}

\usepackage{amssymb,amsmath,amsthm,amsfonts,amsopn,url,color,
enumerate,mathtools,microtype,MnSymbol,comment,array,nicematrix,pifont}
\newcommand{\cmark}{\ding{51}}
\newcommand{\xmark}{\ding{55}}
\usepackage{mathrsfs}
\usepackage{multirow}
\usepackage{multicol}
\usepackage{enumitem}
\usepackage[colorlinks=true,linkcolor=red,urlcolor=green,citecolor=teal,pagebackref=true]%
{hyperref}
\usepackage[margin=2.5cm]{geometry}
\usepackage[normalem]{ulem}
\usepackage{xcolor}
\usepackage[all]{xy}
\input xy
\xyoption{all}
\usepackage{amscd}
\usepackage{soul}

\theoremstyle{plain}
\newtheorem{thm}{Theorem}[subsection]

\newtheorem{prop}[thm]{Proposition}
\newtheorem{lemma}[thm]{Lemma}
\newtheorem{cor}[thm]{Corollary}

\theoremstyle{definition}
\newtheorem{defn}[thm]{Definition}
\newtheorem{plaindef}{Definition}[section]
\newtheorem{plainrem}[plaindef]{Remark}
\newtheorem*{defn*}{Definition}
\newtheorem*{question*}{Question}

\newtheorem{example}[thm]{Example}
\newtheorem*{example*}{Example}
\newtheorem{rem}[thm]{Remark}
\newtheorem*{rem*}{Remark}
\newtheorem{notation}[thm]{Notation}

\newcommand{\field}[1]{\mathbb{#1}}
\newcommand{\N}{\field{N}}
\newcommand{\Z}{\field{Z}}

\newcommand{\ideal}[1]{\mathfrak{#1}}
\newcommand{\m}{\ideal{m}}
\newcommand{\n}{\ideal{n}}
\newcommand{\p}{\ideal{p}}

\newcommand{\ia}{\ideal{a}}

\newcommand{\func}[1]{\mathrm{#1} \,}

\newcommand{\coker}{\func{coker}}
\newcommand{\im}{\func{im}}

\newcommand{\arrow}[1]{\stackrel{#1}{\rightarrow}}

\newcommand{\ra}{\rightarrow}

\DeclareMathOperator{\ann}{ann}

\DeclareMathOperator{\Hom}{Hom}
\newcommand{\intr}{\mathrm{int}}

\newcommand{\be}{\begin{enumerate}}
\newcommand{\ee}{\end{enumerate}}

\newcommand{\li}
 {\leftfootline}
\newcommand{\lic}[2]{{#1}^{\li}_{#2}}

\newcommand{\Ri}{{-\rm{Rs}}}

\newcommand{\EHUi}{{-\rm{EHU}}}

\newcommand{\onto}{\twoheadrightarrow}

\newcommand{\into}{\hookrightarrow}

\newcommand{\cA}{\mathcal{A}}
\newcommand{\cC}{\mathcal{C}}

\newcommand{\cD}{\mathcal{D}}

\newcommand{\cM}{\mathcal{M}}
\newcommand{\cP}{\mathcal{P}}

\newcommand{\cR}{\mathcal{R}}
\newcommand{\cS}{\mathcal{S}}
\renewcommand{\phi}{\varphi}

\DeclareMathOperator{\chr}{char}

\DeclareMathOperator{\Sym}{Sym}

\newcommand{\opsub}{order-preserving on submodules}

\newcommand{\dual}{\smallsmile}
\newcommand{\cl}{{\mathrm{cl}}}
\newcommand{\rp}{{\mathrm{p}}}
\newcommand{\riq}{{\mathrm{q}}}
\let\int\relax
\DeclareMathOperator{\int}{i}

\DeclareMathOperator{\ch}{c}
\DeclareMathOperator{\h}{h}
\newcommand{\po}[2]{{#1}_{#2}}

\DeclareMathOperator{\tr}{tr}

\DeclareMathOperator{\core}{-core}
\DeclareMathOperator{\cre}{core}
\DeclareMathOperator{\hull}{-hull}

\newcommand{\Jcolsym}[1]{{#1} \rm{bf}}

\newcommand{\Jcol}[3]{{#2}^{\Jcolsym #1}_{#3}}

\newcommand{\Jintrelsym}[1]{{#1}{\rm be}}

\newcommand{\Jintrel}[3]{{#3}_{\Jintrelsym{#1}}^{#2}}

\newcommand{\fg}{finitely generated}
\newcommand{\charp}{characteristic $p>0$}
\newcommand{\CM}{Cohen-Macaulay}

\newcommand{\bemp}{basically empty}

\newcommand{\cmg}{\color{magenta}}
\newcommand{\cbl}{\color{blue}}

\newcommand{\perf}{{\mathrm{perf}}}
\newcounter{rfs}
\DeclareRobustCommand{\nextref}[1]{{\refstepcounter{rfs}\label{#1}}}

\author{Neil Epstein}
\address{Department of Mathematical Sciences \\ George Mason University \\ Fairfax, VA  22030}
\email{nepstei2@gmu.edu}

\author{Rebecca R.G.}
\address{Department of Mathematical Sciences \\ George Mason University \\ Fairfax, VA  22030}
\email{rrebhuhn@gmu.edu}

\author{Janet Vassilev}
\address{Department of Mathematics and Statistics \\ University of New Mexico \\ Albuquerque, NM 87131}
\email{jvassil@math.unm.edu}
\title[Closure, test ideals, and duality framework]{A common framework for test ideals, closure operations, and their duals}

\date{March 24, 2026}

\begin{document}

\begin{abstract}

     Closure operations such as tight and integral closure and test ideals have appeared frequently in the study of commutative algebra. This articles serves as a survey of the authors' prior results connecting closure operations, test ideals, and interior operations via the more general structure of pair operations. Specifically, we describe a duality between closure and interior operations generalizing the duality between tight closure and its test ideal,  provide methods for creating pair operations that are compatible with taking quotient modules or submodules, and describe a generalization of core and its dual. Throughout, we discuss how these ideas connect to common  constructions in commutative algebra.
\end{abstract}

\maketitle

\section{Introduction}

Closure operations have been important in commutative algebra for a long time (see \cite{nme-guide2}). For example, the integral closure and minimal reductions of Northcott and Rees \cite{NR} and the prime operations of Krull \cite{Kr-domains1} led to an extended study of integral closures of ideals and modules, including multiple books on the subject \cite{HuSw-book,Vasc-icbook}. 
Many classes of rings can be described through classes of integrally closed ideals; for example, Noetherian rings whose principal ideals are integrally closed are normal rings \cite[Proposition 1.5.2]{HuSw-book} and Pr\"ufer domains (Dedekind domains in the Noetherian setting) are the domains whose ideals are all integrally closed \cite{Jensen-Prufer}. The Brian\c{c}on-Skoda Theorem, a theorem whose roots lie in analysis and determines the power of an ideal $I$ whose integral closure lies in $I$, has inspired many results in commutative algebra including its tight closure variant \cite[Chapter 13]{HuSw-book}.  

In the past 35 years, operations like tight closure \cite{HHmain} have been a key tool in studying singularities of commutative rings \cite{ScTu-survey,Sm-param,Di-clCM,MaSchwedePerfectoidBCMSings}. The discovery that the uniform annihilator of tight closure can often be realized as the annihilator of a single tight closure module in the injective hull of the residue field \cite{HHmain,ScTu-survey}, along with work connecting this ``test ideal'' to the multiplier ideals used by algebraic geometers \cite{Sm-param,HaTa-gentest}, further cemented the usefulness of closure operations to commutative algebra and algebraic geometry.

The latter work led the first named author and Karl Schwede \cite{nmeSc-tint} to the first steps toward a closure-interior duality in the form of the \emph{tight interior} operation.  This in turn led the second named author and Felipe P\'erez \cite{PeRG} to describe a duality between module closures and trace ideals that is parallel to the duality between tight closure and its test ideal. We \cite{nmeRG-cidual, ERGV-chdual, ERGV-nonres, ERGV-extend} then built a general theory of how closure operations and their test ideals can be viewed as dual, applying the theory to other examples such as integral closure (building on \cite{EHU-Ralg, nmeUlr-lint}) and basically full closure (as in \cite{HRR-bf}), and generalizing it along the way to a much broader duality of pair operations. Since the ultimate foundation goes back to the work of Hochster and Huneke on tight closure, it seemed appropriate to place this article into a volume dedicated to the work of these two giants of commutative algebra.

We developed our theory of duality between pair operations over the course of four papers totaling over 150 pages \cite{nmeRG-cidual, ERGV-chdual, ERGV-nonres, ERGV-extend}. It became clear over that period that it would be useful to write a survey, as a kind of guide to the ideas and techniques and to give new readers a quicker entry point to our results.  This has the additional advantage of standardizing our notation, which grew over time and hence is inconsistent between papers. Additionally, some of our early results stated before we began working in the generality of pair operations can easily be extended to this more general construct.  We therefore use this survey paper as an opportunity to address the above issues.

 In order to motivate general pair operations, it makes sense to zoom out and consider what closure and interior operations really are. Consider the following scenario: Given an $R$-submodule inclusion $L \subseteq M$ defined in some abstract way, where $M$ is well-understood but $L$ is not, a basic problem is to figure out which elements of $M$ are in $L$. 

One may sometimes approximate from \emph{above}.  That is, one finds a \emph{necessary} condition for an element $x\in M$ to be in $L$, so that if $x$ fails the condition, it cannot be in $L$.  This often leads to the notion of a \emph{closure operation} on submodules of $M$.  One decides which submodules of $M$ are \emph{closed}, and then the \emph{closure} $L^\cl_M$ of $L$ in $M$ is the intersection of all such submodules that contain $L$.  A typical example is given by letting $M=R$, and the closure of an ideal its radical; here the prime ideals make a basis for the closed ideals in $R$ (i.e. every closed ideal is an intersection of them, and vice versa).  The condition  for containment in the radical of an ideal is that some power of the element is in the ideal $L$.  Other important examples of closure operations include tight closure \cite{HHmain} and various versions of integral closure \cite{Rees-redmod}, \cite{EHU-Ralg}, \cite{SUVIntClosure},  \cite{nmeUlr-lint}.

However, sometimes an approximating condition from above is not really a closure operation, in the sense that it may not be idempotent -- i.e. the ``closure'' isn't closed.  For instance, if $(R,\m)$ is a local ring, then taking the \emph{socle} of $L$ in $M$, given by $(L :_M \m)$, i.e., the elements of $M$ that $\m$ multiplies into $L$, is not an idempotent operation.  Perform the operation a second time and you get the elements that $\m^2$ multiplies into $L$.  A related important example is the $\ia$-tight closure of Hara and Yoshida, which is dual to the test ideal for pairs \cite{HaYo-atc}. 
Alternately, such an operation might not be order-preserving -- i.e. one may have $K \subseteq L$ but the ``closure'' of $K$ isn't in the ``closure'' of $L$.  The Ratliff-Rush operation \cite{RatRu-rr} exhibits this 
pathology   \cite[Example 1.11]{HeLaShRR}, \cite[1.1]{HJLS-coeff}.

A less common but also useful thing is to approximate from \emph{below}.  That is, one finds a \emph{sufficient} condition for an element $x\in M$ to be in $L$, so that if $x$ satisfies the condition, it must be in $L$.  This can lead to the notion of an \emph{interior operation} on submodules of $L$.  Here, one decides which submodules of $M$ are \emph{open}, and then the \emph{interior} $L_\intr^M$ of $L$ with respect to $M$ is the sum of all such submodules of $M$ that are also submodules of $L$.  A typical example is given by letting $R$ and $M$ be graded, but $L$ not necessarily graded, and the interior of $L$ in $M$ the sum of all homogeneous submodules of $L$.  Here the cyclic homogeneous submodules of $M$ make a basis   for the open submodules (i.e., every open submodule is a sum of them).  The condition that determines containment in this interior is that $x$ is a sum of homogeneous elements of $L$.  Other examples include \emph{tight interior} \cite{nmeSc-tint} and \emph{basically empty interior} (See Definition~\ref{def:bee}.)
However, not all approximations from 
below lead to interior operations.  For example, the $\cl$-core of a submodule $L$ of $M$ is the intersection of all $\cl$-reductions of $L$ in $M$ \cite{FoVa-core}, \cite{ERGV-chdual}; although $\cl$-core$_M(L) \subseteq L$, this operation is typically not idempotent (see for example \cite[Theorem 4.4]{FPUgradanncore}),  nor is it order-preserving on submodules \cite{Lee-core}. 

In our previous work, we established a link between closure and interior operations when $(R,\m,k)$ is a complete Noetherian local ring through Matlis duality.  The simplest version of this link occurs when $M=R$ and $\cl$ is a closure operation on submodules of the injective hull $E = E_R(k)$ of the residue field.  Then one obtains an interior operation $\intr$ on the ideals of $R$ by setting $I_\intr := \ann_R((\ann_EI)^\cl_E)$.  Interestingly, if $\intr'$ is an interior operation on submodules of $E$, then one obtains a closure operation $\cl'$ on the ideals of $R$ in the same way, by setting $I^{\cl'} := \ann_R((\ann_EI)_{\intr'}^E)$. See Proposition \ref{pr:test} and Theorem \ref{thm:finintideal} for more details.

The above considerations have led us to consider pair operations in general, where given a submodule inclusion $L \subseteq M$, one obtains a submodule $p(L,M)$ of $M$ in some systematic way
.  This then encompasses closure operations, other necessary conditions for submodule containment, interior operations, other sufficient conditions for submodule containment, etc.  In the complete local case, we defined and developed a duality operator on  general pair operations, called the \emph{smile dual} $p \mapsto p^\dual$ (see Section \ref{sec:dual}).  The interior-closure duality outlined  in the preceding paragraph is a special case of this smile duality.  When the modules in question are Artinian or Noetherian, one obtains $p^{\dual \dual} = p$ (see Lemma \ref{lem:doubledual}).  In prime characteristic, the smile dual of tight closure is tight interior as in \cite{nmeSc-tint}, which in turn generalizes test ideals, in the sense that $R_{*^\dual}^R$ is the big test ideal of $R$.

An advantage to this framework is that it allows one to consider relationships between different kinds of operations in sophisticated ways. 
This duality creates a correspondence between properties of a closure, interior, or pair operation and properties of its dual (see Table~\ref{tab:dual}). 
For instance, $p$ is idempotent (i.e. $p\circ p = p$) if and only if $p^\dual$ is idempotent.  A pair operation $p$ is \emph{extensive} (i.e. $L \subseteq p(L,M)$ for all relevant $L\subseteq M$) if and only if $p^\dual$ is \emph{intensive} (i.e. $p^\dual(A,B) \subseteq A$ for all relevant $A \subseteq B$).

We begin with a discussion of pair operations and their common properties (Section \ref{sec:pairs}), followed by results on the smile dual operation that sends closure operations to interior operations and vice versa (Section \ref{sec:dual}). Section \ref{sec:version} goes over the residual and hereditary properties of a pair operation, and how to create versions of a pair operation with these properties. Section \ref{sec:subselect} describes how our earliest work on submodule selectors interacts with more recent work on pair operations, building up to a result describing the duality between closure operations and their test ideals (Theorem \ref{thm:finintideal}). In Section \ref{sec:lim}, we summarize results on meets and joins of pair operations, updating our previous work on limits of submodule selectors. In Section~\ref{sec:bigtable}, we provide an extensive table of pair operations from the literature and their properties, 
along with explanations and useful proofs and counterexamples to complement the existing literature.  Then in  Section~\ref{sec:trace}, we demonstrate how our results apply to the context of module closures and trace ideals. Both examples and applications of these appear in many places in the literature, and we will provide examples of both in Section~\ref{sec:trace}. 
This section doubles as an introduction to traces and module closures more generally. Finally, in Section \ref{sec:corehull}, we discuss cores and hulls for arbitrary Nakayama closures, and describe their duality and applications to basically full closure. 

We hope this survey article will be a clean introduction to pair operations and duality as they apply to tight closure, integral closure, test ideals, traces, module closures, and whatever other purpose you, the reader, find for them to serve.

\section{Pair operations}

\subsection{Basic definitions and properties}\label{sec:pairs}




 In this section we define pair operations and give a number of common properties that pair operations may have. The pair operation is the common generalization of closure and interior operations that allows the most flexibility in defining a broader class of operations on modules.

 \begin{notation}
     Throughout the paper, $R$ will be a commutative ring with unity. However, most of the results, definitions, and constructions not using  smile duality will work for more general rings.
 \end{notation}


\begin{defn}[{\cite[Definition 2.2]{ERGV-nonres}, \cite[Definition 2.1, Definition 3.1 and 3.12]{ERGV-extend}}]
\label{def:pairop}
Let $\cM$ be a category of $R$-modules. Let $\cP$ be a collection of pairs $(L,M)$, where $L$ is a submodule of $M$, and $L,M \in \cM$, 
such that whenever $\phi:M \to M'$ is an isomorphism in $\cM$ and $(L,M) \in \cP$, $(\phi(L),M') \in \cP$ as well.

A \emph{pair operation} is a function $p$ that sends each pair $(L,M) \in \cP$ to a submodule $p(L,M)$ of $M$, in such a way that whenever $\phi: M \ra M'$ is isomorphism in $\cM$
 and $(L,M) \in \cP$, then 
  $\phi(p(L,M)) = p(\phi(L),M')$.
\end{defn}


\begin{rem}\label{ourpairs}
    Throughout this paper, our collection of pairs $\cP$ will typically be all pairs of $R$-modules $(L,M)$ with $L \subseteq M$, such that both are in one of the following categories: the category of all $R$-modules, the category of finitely generated $R$-modules, the category of Artinian $R$-modules or the category of Matlis dualizable $R$-modules (when the ring is complete local).
\end{rem} 

The basic properties of pair operations that we need for the definitions of closure and interior operations \cite[Definition 2.2]{ERGV-nonres} are given in  Table~\ref{ta:basics}. For Table~\ref{ta:basics}, $(L,M)$ is an arbitrary pair in $\cP$ and $(N,M)$ is an arbitrary pair such that $L \subseteq N$.

\begin{table}[h]
\fbox{
\begin{tabular}{c|c}
Property name  &  Definition\\
\hline\hline
\emph{idempotent}  & $p(p(L,M),M)=p(L,M)$\\
\hline
\emph{extensive}  &  $L \subseteq p(L,M)$\\
\hline
\emph{intensive} &  $p(L,M) \subseteq L$\\
\hline
\emph{order-preserving}  & $p(L,M) \subseteq p(N,M)$\\
\emph{on submodules} & \\
\end{tabular}}
\caption{First properties}
\label{ta:basics}
\end{table}

\begin{defn} \cite[Definition 2.2]{ERGV-nonres}
  A pair operation $p$ on a class $\cP$ of pairs of $R$-modules $(L,M)$ as above is
  \begin{itemize}
      \item a \emph{closure operation} if it is extensive, order-preserving on submodules, and idempotent;
    \item an \emph{interior operation} if it is intensive, order-preserving on submodules, and idempotent.
  \end{itemize}
\end{defn}


\begin{notation}
Throughout the paper, when $p$ is extensive (e.g., a closure operation) we will write $N_M^p$ for $p(N,M)$, and when $p$ is intensive (e.g., an interior operation), we will write $N_p^M$ for $p(N,M)$.
\end{notation}

The properties in Table~\ref{ta:func} express how pair operations behave with respect to homomorphisms between modules $M$ and $M'$ (See \cite[Definition 3.1]{ERGV-extend} for all except `fully functorial').  For 
Table~\ref{ta:func}, $L \subseteq N \subseteq M$, $K \subseteq M$, and $L' \subseteq M'$ are arbitrary $R$-module inclusions in $\cM$, $\pi: M \onto M'$ an arbitrary epimorphism  in $\cM$, and $g:M \rightarrow M'$ an arbitrary $R$-linear map  in $\cM$.

\begin{table}[h]
\fbox{
\begin{tabular}{c|c}
Property name  & Definition\\
\hline\hline
\emph{order preserving} &   $p(L,N) \subseteq p(L,M)$\\
\emph{on ambient modules} & \\
\hline
\emph{surjection-functorial}  & $\pi(p(L,M)) \subseteq p(\pi(L),M')$\\
\hline
\emph{functorial}  & $g(p(L,M)) \subseteq p(g(L), M')$\\
\hline
\emph{restrictable} & $p(L \cap K,K) \subseteq p(L,M)$\\
\hline
\emph{surjection-cofunctorial} & $p(\pi^{-1}(L'),M) \subseteq \pi^{-1} (p(L', M'))$ \\
\hline
\emph{cofunctorial}  & $p(g^{-1}(L'),M) \subseteq g^{-1} (p(L', M'))$ \\
\hline
\emph{fully functorial} & $g(p(L,M)) \subseteq p(L',M')$
\\
\end{tabular}
}
\caption{Reaction to linear maps}
\label{ta:func}
\end{table}

\begin{rem}
    An equivalent defining property for surjection-functoriality is that for 
    $U \subseteq M$,  \[(p(L,M)+U)/U \subseteq p((L+U)/U, M/U).\]  
Also note that $p$ is functorial if and only if 
$p$ is both  order-preserving on ambient modules and surjection-functorial, by the usual epi-monic factorization. 
See \cite[Definition 2.11]{ERGV-nonres}.
\end{rem} 

\begin{rem}\label{rem:funcforideals}

In our framing, an operation `on ideals' can be considered a pair operation by making the ambient module always be $R$, so that the pairs are those of the form $(I,R)$ (with $I\in \cM$).  Then $p$ is functorial if and only if for any $x\in R$ (such that the homothety map $x: R \ra R$ is in $\cM$) and ideals $I$ ($\in \cM$), $xp(I,R) \subseteq p(xI,R)$.  Similarly, $p$ is cofunctorial if and only if for all such $x$ and $I$, $p((I:_Rx),R) \subseteq (p(I,R):_Rx)$.
\end{rem}

When $p$ is order preserving on submodules then 
 Table~\ref{tab:opsub} illustrates which properties from the above table are equivalent.  The proofs are found in \cite[Proposition 3.5]{ERGV-extend}

\begin{table}[h]
\fbox{\begin{tabular}{c c c c c}
\multicolumn{5}{c}{When $p$ is order preserving on submodules, $p$ is $\ldots$} \\
\hline\hline

order preserving on ambient modules & $\Leftrightarrow$ &restrictable\\
\hline
surjection-functorial & $\Leftrightarrow$ & surjection-cofunctorial \\
\hline
functorial & $\Leftrightarrow$ & cofunctorial & $\Leftrightarrow$ & fully functorial
\end{tabular}}
\caption{Some restricted equivalences}
\label{tab:opsub}
\end{table}

\begin{rem}
The category-minded reader may wonder what the ``functor'' is in the definitions of (co/surjection/fully) functorial.  Let $\cM$, $\cP$ be as in Definition~\ref{def:pairop}.  Then one can view $\cP$ as a category by letting the morphisms $(L,M) \ra (L',M')$ be the $\cM$-morphisms $\phi: M \ra M'$ such that $\phi(L) \subseteq L'$.  Then a pair operation $p$ on $\cP$ becomes a \emph{functor} from $\cP$ to $\cM$ precisely when, for any such $\phi$, we have $\phi(p(L,M)) \subseteq p(L',M')$ -- i.e., $p$ is fully functorial.

Any fully functorial pair operation $p$ is functorial, cofunctorial, and order-preserving on submodules.  
Conversely, if $p$ is order-preserving on submodules (as is the case for many well-studied operations, including all closure and interior operations), and it is \emph{either} functorial or cofunctorial, then it is fully functorial. See Proposition~\ref{pr:func}. As the  property `functorial' was first identified (by Hochster) in the context of closure operations, the original definition and naming scheme makes sense.
\end{rem}

\begin{prop}\label{pr:func}
Let $\cM$ be a category of modules and $\cP$ be a class of pairs in $\cM$, such that whenever $(L,M) \in \cP$ and $M' \arrow{\phi} M \arrow{\psi} M''$ are morphisms in $\cM$, $(\phi^{-1}(L),M'), (\psi(L),M'') \in \cP$. (This holds in Remark~\ref{ourpairs}.)  Let $p$ be a pair operation on $\cP$.  Consider the following conditions: \begin{enumerate}
    \item\label{it:functor} 
    $p$ is fully functorial.
    \item\label{it:functorial} $p$ is functorial.
    \item\label{it:cofunctorial} $p$ is cofunctorial.
    \item\label{it:opsub} $p$ is \opsub.
\end{enumerate}
%
Then (\ref{it:functor}) $\implies$ (\ref{it:functorial}), (\ref{it:cofunctorial}), and (\ref{it:opsub}).  Conversely, (\ref{it:functorial}) $\&$ (\ref{it:opsub}) $\implies$ (\ref{it:functor}), and also (\ref{it:cofunctorial}) $\&$ (\ref{it:opsub}) $\implies$ (\ref{it:functor}).
\end{prop}

\begin{proof}
Suppose (\ref{it:functor}).  Let $(L,M), (K,M') \in \cP$ and $\phi: M \ra M'$ in $\cM$. Then $\phi: (L,M) \ra (\phi(L),M')$ is a morphism in $\cP$, so since $p$ is 
fully functorial, we have $\phi(p(L,M)) \subseteq p(\phi(L),M')$; that is, (\ref{it:functorial}) holds.  Moreover, $\phi:(\phi^{-1}(K),M) \ra (K,M')$ is a morphism in $\cP$, so $\phi(p(\phi^{-1}(K),M)) \subseteq p(K,M')$. Thus, $p(\phi^{-1}(K),M) \subseteq \phi^{-1}(p(K,M'))$; that is, (\ref{it:cofunctorial}) holds.  Now suppose $(N,M) \in \cP$ and $L \subseteq N$.  Then the identity map $1_M: M \ra M$ extends to a morphism $g: (L,M) \ra (N,M)$ in $\cP$, since $1_M(L) = L \subseteq N$.  Thus, $p(L,M) = 1_M(p(L,M)) \subseteq p(N,M)$ by (\ref{it:functor}), whence (\ref{it:opsub}) holds.

Conversely, let $\phi:(L,M) \ra (L',M')$ be a morphism in $\cP$ (i.e. $\phi(L) \subseteq L'$) and suppose (\ref{it:opsub}) holds.  If (\ref{it:functorial}) holds,  then $\phi(p(L,M)) \subseteq p(\phi(L),M')$.  Since (\ref{it:opsub}) holds, $p(\phi(L),M') \subseteq p(L',M')$.  On the other hand if (\ref{it:cofunctorial}) holds, then $p(\phi^{-1}(L'), M) \subseteq \phi^{-1}(p(L',M))$.  Since (\ref{it:opsub}) holds and $L \subseteq \phi^{-1}(L')$, we have $p(L,M) \subseteq p(\phi^{-1}(L'),M)$.  In either case, it follows from transitivity of set containment that $\phi(p(L,M)) \subseteq p(L',M')$; that is, (\ref{it:functor}) holds.
%
\end{proof}

\begin{defn}
\label{def:persistandLS}
Let $\cR$ be a category of rings, and for each $R \in \cR$, let $\cM_R$ be a category of $R$-modules and $\cP_R$ be a category of pairs as above (see Definition~\ref{def:pairop}).  For an $R$-submodule inclusion $K \subseteq M$ and a ring map $\phi:R \ra S$, define $KS$ to be the image of the map $K \otimes_R S \ra M \otimes_R S$ obtained by applying the functor $- \otimes_R S$ to the inclusion map $K \into M$. Assume that for any ring map $\phi:R \ra S$ in $\cR$ and any $(L,M) \in \cP_R$, we have $(LS, M \otimes_R S) \in\cP_S$. 

We say that $p$ is \emph{persistent} on $\cR, \cP$ if for any $\phi: R \ra S$ in $\cR$ and any $(L,M) \in \cP_R$, we have $p_R(L,M)S \subseteq p_S(LS, M \otimes_RS)$.
\end{defn}

Despite the similarity in definitions, persistence and functoriality are very different properties.  For example, it is elementary that functoriality of tight closure operation holds in full generality, but \emph{persistence} of tight closure fails in general \cite{Heit-tcnon}. Persistence of tight closure does hold in the cases most people care about, but this is a hard theorem (see e.g. \cite[Theorem 6.24]{HHbase}).

Inspired by the typical ways that closure operations on the ideals of rings have been extended to modules (see \cite[Section 7]{nme-guide2}, and \cite[Section 6]{nmeRG-cidual}), we defined the following properties of pair operations in \cite[Definition 3.12]{ERGV-extend}.  In Table~\ref{ta:hered}, $(L,N)$ and $(L,M)$ are arbitrary pairs such that $L \subseteq N \subseteq M$, and $\pi: M \onto M/L$ is the canonical epimorphism. 
\medskip

\begin{table}[h]
\begin{center}\fbox{
\begin{tabular}{c|c}
Property name  & Definition\\
\hline\hline
\emph{hereditary} & $p(L,N) = p(L,M) \cap N$\\
\hline
\emph{absolute} & $p(L,N) = p(L,M)$ \\
\hline
\emph{cohereditary} & $p(N/L, M/L) = \displaystyle\frac{p(N,M) + L}L$ \\
& \\
\hline
\emph{residual}  &$p(N,M) = \pi^{-1}(p(N/L, M/L))$ \\
\end{tabular}
}\end{center}
\caption{Different styles of operation}
\label{ta:hered}
\end{table}

\begin{rem}
It would be equivalent to define \emph{absolute} and \emph{hereditary} in the following ways: Given $L \subseteq N$ and an injective map $j: N \hookrightarrow M$, 
\begin{itemize}
    \item absolute: $j(p(L,N)) = p(j(L),M)$.
    \item hereditary: $p(L,N) = j^{-1}(p(j(L),M)$.
    \end{itemize}
This is because by definition, a pair operation is invariant under isomorphism, so we can treat $j$ as an inclusion.
\end{rem}

\begin{rem}
    When $p$ is intensive, hereditary and absolute are the same \cite[Lemma 3.15]{ERGV-extend}.  When $p$ is extensive, cohereditary and residual are the same \cite[Lemma 3.16]{ERGV-extend}.
\end{rem}

We will discuss these properties and their uses further in Section \ref{sec:version}.

\medskip

For local rings $(R, \m )$, we can define \emph{Nakayama} closures and interiors. Nakayama closures can be used to define and prove theorems about a $\cl$-\emph{core} analogous to the integral closure core, as discussed in Section \ref{sec:corehull}. 

\begin{defn}[{\cite[Definitions 2.11 and 3.11] {ERGV-nonres}, \cite[Definition 2.6]{ERGV-extend}}]
\label{def:nakayama}
Let $(R,\m)$ be a Noetherian local ring.

Let $\cl$ be a closure operation on the class of pairs of \fg\ $R$-modules.  We say that $\cl$ is a \emph{Nakayama closure} if for \fg\ $R$-modules $L \subseteq N \subseteq M$, if $L\subseteq N \subseteq (L+\m N)^{\cl}_M$ then $L^{\cl}_M=N^{\cl}_M$.

Let $\intr$  be an interior operation on the class of pairs of Artinian $R$-modules. We say that $\int$ is a \emph{Nakayama interior} if for any Artinian $R$-modules $A \subseteq C \subseteq B$, if 
$(A:_C \m)^B_{\intr}\subseteq A$, then $A^B_{\intr} = C^B_{\intr}$ (or equivalently, $C^B_{\intr} \subseteq A$).
\end{defn}

Integral closure \cite[Page 372]{nme*spread}, tight closure  \cite[Proposition 2.1]{nme*spread} and Frobenius closure \cite[Lemma2.2 and Proposition 4.2]{nme-sp} are all Nakayama closures on local rings.  The tight interior as defined in \cite{nmeSc-tint} is a Nakayama interior \cite[Proposition 5.5]{ERGV-chdual}.  In fact, any interior which is dual to a Nakayama closure is a Nakayama interior \cite[Proposition 3.12]{ERGV-nonres}.

\subsection{The dual of a pair operation}\label{sec:dual}

In this section, we detail our central concept, the duality on pair operations that sends closure operations to interior operations and vice versa. 

\begin{rem}
\label{rem:matlisdualizable}
    Throughout this section (and whenever we use duality in the papers), we assume that our ring is complete Noetherian commutative local and our category $\cP$ of pairs only contains Matlis-dualizable modules.
\end{rem}


\begin{defn}
\label{def:nonresidualdual}
Let $(R,\m,k)$ be a complete local ring and $E := E_R(k)$ the injective hull of the residue field. Let $p$ be a pair operation on a class of pairs  $\cP$ as in Remark \ref{rem:matlisdualizable}. 
Set $\cP^\vee := \{(A,B) \mid ((B/A)^\vee, B^\vee) \in \cP\}$, and $\eta_B: B \ra B^{\vee \vee}$ the Matlis duality isomorphism $x \mapsto (g \mapsto g(x))$, and define the dual of $p$ by 
\[
p^\dual(A,B):=\eta_B^{-1}\left(\left(\frac{B^\vee}{p\left(\left(\frac{B}{A}\right)^\vee,B^\vee\right) } \right)^\vee \right).
\]
 In the above, we use the convention that when $N \subseteq M$ are Matlis-dualizable modules, then $(M/N)^\vee$ is identified with the submodule $\{g \in M^\vee \mid g(N)=0\}$ of $M^\vee$.

We leave it as an exercise to show that $(\cP^\vee)^\vee=\cP$.
\end{defn}


\begin{lemma}\label{lem:doubledual} \cite[Proposition 3.6 (1)]{ERGV-nonres}
    Under the hypotheses of Remark \ref{rem:matlisdualizable}, 
    $p^{\dual \dual} =p. $ 
\end{lemma}


The following result adapts Theorem 3.3 of \cite{nmeRG-cidual} to the pair operations setting, giving another way to understand the dual of a pair operation:

\begin{prop}
\label{pr:kernelview}
    Let $R$ be a complete local ring and $p$ a pair operation on a class of pairs $\cP$ as in Remark~\ref{rem:matlisdualizable}.
    Let $x \in M$ and $(L,M)$ a pair.  Then $x \in p^\dual(L,M)$ if and only if $x \in \ker(g)$ for every $g \in p((M/L)^\vee,M^\vee)$.
\end{prop}

\begin{proof}
Write $P = p((M/L)^\vee, M^\vee)$.  We have \begin{align*}
x \in p^\dual(L,M) &\iff \eta_M(x): M^\vee \ra E \text{ vanishes on }P \\
&\iff \forall g\in P\text{, } 0 = \eta_M(x)(g) = g(x). \qedhere
\end{align*}

%
%
\end{proof}


Typically, a property a pair operation may have is equivalent to its smile dual having another property.  
For instance, $p$ is a closure operation if and only if $p^\dual$ is an interior operation. Table~\ref{tab:dual} is a chart of many such correspondences. 


\begin{table}[h]
    \begin{center}{\fbox{
\begin{tabular}{>{\centering\arraybackslash}p{3.5cm}|>{\centering\arraybackslash}p{3.5cm}|>{\centering\arraybackslash}p{4cm}}
$p$ is (a): & $p^\dual$ is (a/an): & citation: \\
\hline\hline
closure operation & interior operation & \small{\cite[Prop. 3.6(6\&7)]{ERGV-nonres}}\\
\hline
idempotent & idempotent & \small{\cite[Prop. 3.6(5)]{ERGV-nonres}}\\
\hline
extensive & intensive & \small{\cite[Prop. 3.6(2\&3)]{ERGV-nonres}}\\
\hline
order-preserving on submodules & order-preserving on submodules & \small{\cite[Prop. 3.6(4)]{ERGV-nonres}}\\
\hline
surjection-functorial & restrictable  & \small{\cite[Prop. 3.6(8)]{ERGV-nonres}}\\
\hline
functorial & cofunctorial & \small{\cite[Prop. 3.9]{ERGV-extend}} \\
\hline
surjection-cofunctorial & order-preserving on ambient modules & \small{\cite[Rmk. 3.10(2)]{ERGV-extend}}\\
\hline
hereditary & cohereditary &\small{\cite[Prop. 3.17(2\&3)]{ERGV-extend}} \\
\hline
residual & absolute & \small{\cite[Prop. 3.17(4)]{ERGV-extend}}\\
\hline
Nakayama closure & Nakayama interior & \small{\cite[Prop. 3.12]{ERGV-nonres}}
\end{tabular}
    }}
    \end{center}
    \caption{Dual Correspondences}
    \label{tab:dual}
\end{table}

It is instructive to see details for an example of this equivalence:
\begin{prop}
    Let $p$ be a pair operation on a class of pairs $\cP$ as in Remark~\ref{rem:matlisdualizable}.
    Then $p$ is residual if and only if $p^\dual$ is absolute.    
\end{prop}
\begin{proof}
First suppose $p$ is residual.  Let $(L,M)$ be a pair with $L \subseteq M$; it is enough to show $p^\dual(L,L) = p^\dual(L,M)$.  Let $x\in p^\dual(L,L)$, and let $\pi: M^\vee \onto L^\vee$ be induced by the inclusion map $L \hookrightarrow M$.  Note that $\pi$ amounts to the restriction map $f \mapsto f|_L$. Then by Proposition~\ref{pr:kernelview}, for all $g \in p(0, L^\vee)$, we have $g(x) = 0$.  Now let $f \in p((M/L)^\vee, M^\vee) = \pi^{-1}(p(0, L^\vee))$ by residuality.  Then $(\pi(f))(x) = (f|_L)(x) = f(x) = 0$.  Since $f$ was arbitrary, $x \in p^\dual(L,M)$ by Proposition~\ref{pr:kernelview}.

For the opposite inclusion, let $x \in p^\dual(L,M)$ and $g \in p(0,L^\vee)$.  By surjectivity of $\pi$, we can choose $f\in M^\vee$ with $g=\pi(f) = f|_L$.  Then by residuality, $f \in \pi^{-1}(p(0,L^\vee)) = p((M/L)^\vee, M^\vee)$, so by Proposition~\ref{pr:kernelview}, $f(x)=0$.  Thus, $g(x) = (f|_L)(x) = 0$, so that $x \in p^\dual(L,L)$ by Proposition~\ref{pr:kernelview}.

For the reverse implication, by Lemma~\ref{lem:doubledual} it is equivalent to prove that if $p$ is absolute, then $p^\dual$ is residual.  Accordingly, suppose $p$ is absolute. Let $L \subseteq N \subseteq M$, and let $\pi: M\onto M/L$ and $q: M/L \onto M/N$ be the natural maps.  We want to show that $p^\dual(N,M) = \pi^{-1}(p^\dual(N/L, M/L))$.  For this, consider the injective maps $i=q^\vee: (M/N)^\vee \into (M/L)^\vee$ and $j = \pi^\vee: (M/L)^\vee \into M^\vee$.

Let $x \in p^\dual(N,M)$.  Let $g \in p(i((M/N)^\vee), (M/L)^\vee)$.  Then by absoluteness, we have $j(g) = g\circ \pi \in p(j(i((M/N)^\vee)), M^\vee)$. 
Then by Proposition~\ref{pr:kernelview}, $g(\pi(x)) = 0$.  Thus, $\pi(x) \in p^\dual(N/L, M/L)$.

For the opposite inclusion, let $x \in \pi^{-1}(p^\dual(N/L, M/L))$.  That is, $\pi(x) 
\in p^\dual(N/L, M/L)$.  Let $g \in p(j(i((M/N)^\vee)), M^\vee)$.  Then by absoluteness, we have $g = j(f) = f\circ \pi$ for some $f \in p(i((M/N)^\vee), (M/L)^\vee)$.  Thus, $0 = f(\pi(x)) = g(x)$, so $x\in p(N,M)$ by Proposition~\ref{pr:kernelview}.
\end{proof}

Having established the above results, we can make the connection between pair operation duals and 
 annihilators 
explicit.   We will expand on this idea further in Theorem \ref{thm:finintideal}, when we will discuss how it connects to test ideals.

\begin{prop}\label{pr:test}
Let $(R,\m,k)$ be a complete Noetherian local ring and $E = E_R(k)$.  Let $p$ be a pair operation on submodules of $E$. 
Then for any ideal $I$ of $R$,  $p^\dual(I,R) = \ann_R (p(\ann_E(I),E))$.  In particular, $p^\dual(R,R) = \ann p(0,E)$.

If $p$ is defined at least on all pairs of Artinian modules (rather than just those of the form $(D,E)$) and is
functorial, then $p^\dual(R,R)$ multiplies $p(A,B)$ into $A$ for every pair $(A,B)$ on which $p$ is defined, so long as $B/A$ is Artinian.
\end{prop}

\begin{proof}
By Proposition~\ref{pr:kernelview}, for $x\in R$, we have $x \in p^\dual(I,R)$ if and only if $x \in \ker g$ for every $g \in p((R/I)^\vee,R^\vee)$.  But under the isomorphism $\mu: R^\vee \stackrel{\cong}{\rightarrow} E$ sending $g \mapsto g(1)$ and by the identification of $(R/I)^\vee$ as the submodule of $R^\vee$ of functions that vanish on $I$, we have $\mu((R/I)^\vee) = \ann_E(I)$. So since pair operations are isomorphism-invariant, it follows that $x \in p^\dual(I,R)$ if and only if $x$ annihilates $p(\ann_E(I),E)$.  The second statement follows since $\ann_E(R) =0$. 

For the third statement, let $\pi: B \onto B/A$ be the canonical surjection.  Since $B/A$ is Artinian, there is an injective map $i: B/A \hookrightarrow E^n$ for some $n \in \N$.  For each $1 \leq j \leq n$, let $q_j: E^n \onto E$ be the projection onto the $j$th factor.  Then for any $z \in p(A,B)$, we have $(q_j \circ i \circ \pi)(z) \in p(0,E)$ by functoriality, so that for any $r \in p^\dual(R,R)$, we have by the above that $0=r \cdot (q_j \circ i \circ \pi(z)) = (q_j \circ i \circ \pi)(rz)$. By properties of direct products (since the equation holds for all $j$), it follows that $i(\pi(rz)) = 0$, so that $i$ being injective shows that $\pi(rz) = 0$, so that $rz \in A$.
\end{proof}


\begin{rem}
    In particular, when the pair operation $p$ is a closure operation $\cl$ (resp. an interior operation $\intr$) on submodules of $E$, we get a dual interior (resp. closure) operation on ideals of $R$ such that $I_{\intr}:=\ann_R((\ann_E I)_E^{\cl})$ (resp. $I^\cl:=\ann_R((\ann_E I)_\intr^E)$.
\end{rem}

\subsection{Transforming operations into more amenable versions}
\label{sec:version}

 In this section we discuss several methods of forming new pair operations from known pair operations.

 \begin{rem}
 Given a class $\cP$ of pairs on a category $\cM$ of modules, then assuming $\cM$ is closed under submodule sums and intersections, the set $\cS$ of all pair operations on $\cP$ makes a \emph{complete lattice}, as follows.

 Define a relation $\leq$ on $\cS$ by saying $p\leq p'$ if and only if for all pairs $(L,M) \in \cP$, we have $p(L,M) \subseteq p'(L,M)$.  It is immediate that $\leq$ satisfies the conditions for a partial order on $\cS$.  Moreover, for any collection $\{p_\alpha\}_{\alpha \in \Lambda}$ in $\cS$, the meet $\bigwedge_\alpha p_\alpha$ for this partial order is given by $(\bigwedge_\alpha p_\alpha)(L,M) := \bigcap_{\alpha \in \Lambda} p_\alpha(L,M)$ for all $(L,M) \in \cP$, and the join $\bigvee_\alpha p_\alpha$ is given by $(\bigvee_\alpha p_\alpha)(L,M) := \sum_{\alpha \in \Lambda} p_\alpha(L,M)$.
 \end{rem}



\begin{defn}[c.f. {\cite[Definition~3.7]{ERGV-nonres}, \cite[Definition~3.26]{ERGV-extend}}]
\label{def:finitistic}
Let $R$ be a Noetherian ring, and let $p$ be a pair operation on a class of pairs $\cP$ as in Remark \ref{ourpairs}.
We define the \textit{finitistic version} $p_f$ of $p$ to be 
\[p_f(L,M)=\bigcup \{p(L\cap U,U) \mid U \subseteq M \text{ is \fg\ and } (L \cap U,U) \in \cP\}.\]
We say that $p$ is \textit{finitistic} if for every $(L,M) \in \cP$, $p=p_f$.
\end{defn}

\begin{rem}
    We note as in Lemma 3.8 of \cite{ERGV-nonres} and Lemma 3.2 of \cite{ERGV-chdual} that if a pair operation is functorial and residual, this definition coincides with taking 
    \[p_f(L,M)=\bigcup \left\{p(L,N) \mid L \subseteq N \subseteq M \text{ and } N/L \text{ is finitely-generated} \right\}.\]
\end{rem}

The above is a direct generalization of the notion of \emph{finitistic tight closure} $L^{*fg}_M$ of a submodule \cite[Definition 8.19]{HHmain}, which is closely connected to the open question of whether weakly F-regular rings must be strongly F-regular \cite{HH-sFreg,LySmFreg}.


\begin{rem}
\label{rem:setup}
    The following Proposition  is a version of \cite[Proposition 4.3]{ERGV-extend}. However, to avoid the added assumptions that this general version necessitated in \cite[Notation 4.1]{ERGV-extend}, we assume that:
\begin{enumerate}
    \item $\cM$ is the category of finitely generated $R$-modules,
    \item $\cP$ is the collection of pairs $(L,P)$ with $L \subseteq P$ in $\cM$ and $P$ projective, and
    \item $\cP' \supseteq \cP$ is the collection of pairs $(L,M)$ with $L\subseteq M$  in $\cM$. 
\end{enumerate}    
Then there is necessarily some projective module $P$ with $\pi:P \onto M$ so that $(\pi^{-1}(L),P) \in \cP$. 
\end{rem}

\begin{prop} 
\label{pr:definecohereditaryversion}
Let  $p$ be a cofunctorial pair operation defined on $\cP$ as in Remark \ref{rem:setup}.
Then we can define a cohereditary, cofunctorial pair operation $\po p {\ch}$ on $\cP'$ as follows: for a pair $(L,M) \in \cP'$, let $\pi:P \to M$ be a surjection in $\cM$ with $P$ projective such that $(\pi^{-1}(L),P) \in \cP$. Define
\[\po p {\ch} (L, M) :=\pi(p(\pi^{-1}(L),P)).\]

 In particular, if $p$ is a cohereditary, cofunctorial  pair operation defined on $\cP'$, then $\po p {\ch}=p$.
\end{prop}

The cohereditary version of a pair operations can, in fact, give a new pair operation. The following proposition shows that $\po{p}{\ch}$, if different from $p$,  is generally smaller.

\begin{prop}\label{pr:sfcohereditarypairineq} \cite[Propostion 4.10]{ERGV-extend}
Let $\cP$ and $\cP'$ be as in Remark \ref{rem:setup}, $p$ a cofunctorial pair operation defined on $\cP'$, and $\po p {\ch}$ its cohereditary version on $\cP'$.  Then $\po p {\ch} \leq p$.
\end{prop}

\begin{rem}
    If $p$ is extensive, then $\po p{\ch}$ is residual, so we call it the \emph{residual version} of $p$.  Moreover, if $p$ is a residual, cofunctorial pair operation defined on $\cP'$ then $\po p {\ch}=p$ \cite[Lemma 3.16, Corollary 4.5]{ERGV-extend}. 
\end{rem}

Tight closure for modules \cite{HHmain} was constructed as a residual operation, but two of the versions of integral closure for modules  (See \cite{Rees-redmod} and \cite{EHU-Ralg}) are not residual closures.  One major inspiration for our work was  the liftable integral closure defined by the first-named author and Ulrich \cite{nmeUlr-lint}, which \emph{is} a residual version of integral closure. For another example, the $J$-basically full closures (see Section \ref{sec:corehull}) are not residual, and hence taking the residual versions of these closures gives new smaller closure operations. 

The dual notions to cohereditary and residual versions of pair operations,  namely the hereditary and absolute versions, are developed in \cite[Section 5]{ERGV-extend}. They are defined on the dual category of Artinian modules, replacing projectives, $P$, with injectives, $E$, and replacing the projections $\pi:P \onto M$ with inclusions $i: M \into E$. Both the construction and the proofs are the Matlis duals of the results of this section.  We also show that these new versions give us closure and interior operations that are bigger than the their original versions. For a pre-existing example, absolute tight closure \cite[Section 8]{HHmain} (though fed through a finitistic version first), is otherwise like our hereditary version of tight closure. However, although liftable integral closure is the residual version of EHU-integral closure, in \cite[Example 9.9]{ERGV-extend} we show that the hereditary version of liftable integral closure is \emph{not} EHU-integral closure.

We will not detail these versions here, but we will discuss pair operations induced by pre-enveloping classes as this is how the EHU-integral closure was developed for modules \cite{EHU-Ralg}.  Strikingly, when the pre-enveloping class is the class of projective modules, the hereditary version of a closure agrees with the version of the closure induced by the pre-enveloping class of projective modules.   In particular,  EHU-integral closure is the same as the hereditary version of integral closure \cite[Proposition 9.6, Corollary 9.7]{ERGV-extend}.  Note, however, that the Rees-integral closure is not hereditary \cite[Example 9.8]{ERGV-extend}. 

\begin{defn}
\cite[Definition 6.1.1]{EJ-book}
Let $\cM$ be a category of $R$-modules.  Let $\cC \subseteq \cD$ be two classes of modules in $\cM$.  We say that $\cC$ is a \emph{pre-enveloping class for $\cM$ in $\cD$} if for any $M \in \cM$, there is some $C \in \cC$ and some morphism $\alpha: M \ra C$ in $\cM$, such that for any morphism $g: M \ra D$ in $\cM$ with $D \in \cD$, there is some morphism $\tilde g: C \ra D$ in $\cM$ with $g = \tilde g \circ \alpha$.  In this case the map $\alpha$ (or by abuse of notation, $C$ itself) is called a \emph{$\cC$-preenvelope of $M$ in $\cD$}.
\end{defn}

\begin{prop} \cite[Proposition 8.7]{ERGV-extend}
\label{pr:fgpreenvproj}
If $\cM$ is the category of \emph{finitely generated} $R$-modules, then the class of finitely generated projective $R$-modules is pre-enveloping. 
In fact, the class of finitely generated free modules is pre-enveloping in the class of finitely generated projectives.  
\end{prop}

\begin{rem}
 The pre-enveloping class  $\cC$ of finitely generated projectives is the class we use for the construction below along with integral closure to construct the EHU-integral closure.
\end{rem}

\begin{prop}[See {\cite[Proposition 8.9]{ERGV-extend}}]\label{pr:versalop}
Let $\cM$ be a category of $R$-modules and $\cC$ a pre-enveloping subclass. Let $\cP'=\{(L,M) \mid L,M \in \cM\}$. 
 Let $p$ be a functorial pair operation defined on  $\cP= \{(L,C) \in \cP' \mid C \in \cC\}.$  Define the pair operation $p_{h(\cC)}$ on $\cP'$ so that when $\alpha: M \ra C$ is a $\cC$-preenvelope, $p_{h(\cC)}(L,M) = \alpha^{-1}(p(\alpha(L),C))$. 
Then $p_{h(\cC)}$ is well-defined and functorial.  Moreover,
if $p$ is a closure operation, then so is $p_{h(\cC)}$.
\end{prop}

\subsection{The submodule selector viewpoint} \label{sec:subselect}
 Our first paper on this subject \cite{nmeRG-cidual} presented closure-interior duality through the lens of submodule selectors rather than pair operations.
However, pair operations were lurking in the background. For example, to define residual operations and closure operations, the first two named authors introduced the terminology of \emph{extensive operation} \cite[Definition 2.2]{nmeRG-cidual} which is really the same thing as a pair operation (\cite[Definition 2.2]{ERGV-nonres}) satisfying the extra property of extensivity.  In this section we will show how our results on submodule selectors relate to the context of pair operations.

\begin{defn} \cite[Definition 2.1]{nmeRG-cidual}
    Let $\cM$ be a class of $R$-modules as in Remark~\ref{ourpairs}. A \emph{submodule selector} is a function $\alpha: \cM \rightarrow \cM$ such that  $\alpha(M) \subseteq M$ for all $M \in \cM$ and for any isomorphism of $R$-modules $\phi:M \rightarrow N$ in $\cM$, we have $\phi(\alpha(M))=\alpha(\phi(M))$.
\end{defn}



In Table~\ref{ta:ss}, $L \subseteq M$ is an arbitrary inclusion in $\cM$, $\pi:M \onto Q$ an arbitrary epimorphism, and $g: M \ra N$ an arbitrary $R$-linear map.

\begin{table}[h]
\fbox{
\begin{tabular}{c|c}
Property name  & Definition\\
\hline\hline
\emph{order preserving} &   $\alpha(L) \subseteq \alpha(M)$\\
\hline
\emph{surjection-functorial}  & $\pi(\alpha(M)) \subseteq \alpha(Q)$\\
\hline
\emph{functorial}  & $g(\alpha(M)) \subseteq \alpha(N)$\\
\hline
\emph{idempotent} & $\alpha(\alpha(M))=\alpha(M)$ \\
\hline
\emph{co-idempotent}  & $\alpha(M/\alpha(M))=0$ \\
\end{tabular}
}
\caption{Submodule selector properties}
\label{ta:ss}
\end{table}

Notice that since submodule selectors are intensive by nature, we defined in \cite{nmeRG-cidual} \emph{interior operations} in terms of submodule selectors as those that are order preserving and idempotent.  However, in the latter, pair operation context, these had to be redefined as \emph{absolute} interior operations, excluding as they did the relative ones (i.e., those that depend on both the submodule and the ambient module). For pair operations, we had two notions of order preserving: order preserving on submodules and order-preserving on ambient modules.  
In the table below, when we see how to convert back and forth between submodule selectors and pair operations, it will become clear how both of the pair operation notions of order-preservation correspond to the notion of order-preservation for submodule selectors. 


\begin{defn}[{\cite[Construction 2.3]{nmeRG-cidual} and \cite[Definition 3.21]{ERGV-extend}}]
    Suppose $\alpha$ is a submodule selector and $\pi:M \onto M/L$ is the canonical surjection. We define the \textit{residual operation associated to $\alpha$} by $$r(L,M)=\rho(\alpha)(L,M) := \pi^{-1}(\alpha(M/L))$$ and the \textit{absolute operation associated to  $\alpha$} by $$g(L,M)=\gamma(\alpha)(L,M):=\alpha(L).$$
\end{defn}

\begin{rem}
    In particular, $r$ is a residual pair operation and $g$ is an absolute pair operation, as is clear from the definitions given above.
\end{rem}

 \begin{rem}
Note that these constructions are invertible.  Indeed, given an absolute pair operation $g$, one can define a submodule selector $\alpha$ via $\alpha(M) := g(M,M)$, for which we have $g=\gamma(\alpha)$.  Similarly, given a residual pair operation $r$, one can define a submodule selector $\alpha$ via $\alpha(M) := r(0,M)$, for which we have $r = \rho(\alpha)$.    
\end{rem}

 Table~\ref{ta:sscloint} exhibits some  equivalences between properties of $\alpha$, $r$, and $g$ (\cite[Proposition 2.6]{nmeRG-cidual}, \cite[Proposition 3.24]{ERGV-extend}):
\medskip

\begin{table}[h]
\fbox{
\begin{tabular}{c|c|c}
$\alpha$ has property  & $r$ has property & $g$ has property\\
\hline\hline
order preserving & cofunctorial 
$\iff$& order preserving on submodules   \\
 & order-preserving on ambient modules&$\Leftrightarrow$ restrictable \\
\hline
surjection-functorial & order preserving on submodules  & functorial
\\
\hline
functorial & cofunctorial and & functorial and \\
 & order preserving on submodules &  order preserving on submodules \\
 & $\Leftrightarrow$ functorial and & \\
  & order preserving on submodules  & \\
\hline
idempotent &  & idempotent\\
\hline
co-idempotent & idempotent & \\

\end{tabular}
}
\caption{Passing between submodule selectors and pair operations}
\label{ta:sscloint}
\end{table}


Submodule selectors were an excellent starting point from the viewpoint of tight closure and other module closures, where by definition, the closures were set up as residual closures \cite{HHmain,RG-bCMsing,PeRG}. 
In addition, test ideals were  sometimes defined as annihilators   of  a particular submodule of the injective hull of the residue field, which naturally leads to absolute operations \cite{HHmain,ScTu-survey}. 



\begin{defn}[{\cite[Definition 5.1]{nmeRG-cidual}}]
    Let $\alpha$ be a submodule selector. The \textit{finitistic version} of $\alpha$ is given by
    \[\alpha_f(M):= \sum_{L \subseteq M, L \text{ f.g.}} \alpha(L).\]
\end{defn}

\begin{rem}
    If $\cl$ is a residual, functorial closure operation and $\alpha(M):=0_M^{\cl}$ is the corresponding submodule selector, then  
    $L_M^{\cl_f}$ as defined in Definition \ref{def:finitistic} is equal to $\pi^{-1}(\alpha_f(M/L))$ where $\pi:M \to M/L$ is the standard projection map. 
    
    Further, if $g$ is the absolute operation associated to $\alpha$, then $g_f(L,M)$ as in Definition \ref{def:finitistic} is the same as $\alpha_f(L)$ as defined above.
\end{rem}

The following Theorem is a restatement of \cite[Theorem 3.3]{ERGV-chdual} in the context of closure operations:

\begin{thm}\label{thm:finintideal}
Let $(R,\m,k)$ be a complete Noetherian local ring. 
Let $\cl$ be a functorial, residual closure operation on $\cA$, the category of Artinian $R$-modules, $\cl_f$ its finitistic version, 
and $I$ an ideal of $R$. Then: \begin{align*}
    I^R_{\cl^\dual} &= \ann_R((\ann_E(I))^{\cl}_E) = \bigcap_{M \in \cA} \ann_R((\ann_M(I))^{\cl}_M) \\
    &\subseteq I_{\cl_f^{\dual}}^R = \ann_R((\ann_E(I))^{\cl_f}_E) = \bigcap_{M \subseteq E \text{ f.g.}} \ann_R((\ann_M(I))^{\cl}_M) 
    \\&=
    \bigcap_{\lambda(M) < \infty} \ann_R((\ann_M(I))^{\cl}_M)
    \subseteq \bigcap_{\lambda(R/J)<\infty} \ann_R((\ann_{R/J}(I))^{\cl}_{R/J}) \\ &= \bigcap_{\lambda(R/J)<\infty} (J:(J:I)^{\cl}_R).
\end{align*}
Moreover, the last containment is an equality when $R$ is approximately Gorenstein.  In that case if $\{J_t\}$ is a decreasing nested sequence of irreducible ideals that is cofinal with the powers of $\m$, then in fact we have \[
I_{{\cl_f}^\dual}^R = \bigcap_{t\geq 0} (J_t : (J_t : I)^{\cl}_R).
\]
\end{thm}

\begin{rem}
    The proof works the same way as the proof of the original result, with the same notation substitutions as in the statement of the result above.  The kernel of the result and its proof can be found in Proposition~\ref{pr:test} above.

    In fact, the proof doesn't require idempotence, so the result holds for pair operations that are functorial, residual, extensive, and order-preserving on submodules.
\end{rem}

If we choose $I=R$ here or in Proposition \ref{pr:test}, we get an analogue of the relationship between tight closure and its test ideal. For tight closure, if $(R,\m,k)$ is local,  
the test ideal is defined in the literature both as the ideal generated by elements that always multiply $L_M^*$ into $L$, and as $\ann_R 0_{E_R(k)}^{*}$. The roots of this idea are found in \cite[Theorem 8.23]{HHmain}. 
This relationship is extended to residual, functorial closure operations in \cite{PeRG} and \cite{ERGV-chdual}, using the following definitions:

\begin{defn}\label{def:testideals}
Let $\cl$ be a closure operation on $R$-modules. The (big) $\cl$-test ideal is
\[\tau_{\cl}(R)=\bigcap_{L \subseteq M} L:_R L_M^{\cl}.\]
The finitistic $\cl$-test ideal is
\[\tau_{\cl}^{fg}(R)=\bigcap_{L \subseteq M \text{ f.g.}} L:_R L_M^{\cl}.\]
\end{defn}

The function of the test ideal is to give a single ideal that is larger when the closure operation is close to trivial, and smaller when the closure operation is large. The smaller the test ideal is, the more singular the ring is. 
The test/multiplier ideals of \cite{maschwedetestmultiplierideals} are also examples of this type of construction.




\subsection{Meets, joins, and limits of pair operations}\label{sec:lim}

We discuss meets and joins of posets of pair operations
, explore what properties are preserved under meet and join, 
and show how they are dual. We will later apply this to better understand operations like closure operations coming from a family of modules.

\begin{defn}
    Given pair operations $p,p'$ both defined on the class of pairs $\cP$ as in Remark \ref{ourpairs}, we say that $p \le p'$ if $p(L,M) \subseteq p'(L,M)$ for every pair of $R$-modules $L \subseteq M$ where both are defined. This is a partial order on the set of pair operations.
\end{defn}

\begin{prop}\label{pr:completelattice}
The pair operations on $\cP$ form a \emph{complete lattice}.  Namely, if $\{p_i \mid i \in \Gamma\}$ is a collection of pair operations on $\cP$, then $\bigvee_{i \in \Gamma} p_i$ is given by $(\bigvee_i p_i)(L,M) := \sum_i p_i(L,M)$, and $\bigwedge_{i \in \Gamma} p_i$ is given by $(\bigwedge_i p_i)(L,M) = \bigcap_i p_i(L,M)$.
\end{prop}

\begin{proof}
It is clear that these definitions give pair operations.  To see that they have the right lattice properties, note the following: \begin{itemize}
    \item If $p_i \leq q$ for all $i$, then for all pairs $(L,M)$ we have $\sum_i p_i(L,M) \subseteq q(L,M)$.
    \item For any $j \in \Gamma$, $p_j(L,M) \subseteq \sum_i p_i(L,M)$.
    \item If $p_i \geq q$ for all $i$, then for all pairs $(L,M)$ we have $q(L,M) \subseteq \bigcap_i p_i(L,M)$.
    \item For any $j \in \Gamma$, we have $p_j(L,M) \supseteq \bigcap_i p_i(L,M)$. \qedhere
\end{itemize}
\end{proof}

\begin{defn}[See {\cite[Definitions 7.1 and 7.3]{nmeRG-cidual}} for the submodule selector versions]
Let $\Gamma$ be a directed poset (i.e., for all $i,j \in \Gamma$, there is a $k \in \Gamma$ such that $i,j \le k$), and $\{p_j\}_{j \in \Gamma}$ a set of pair operations defined on the class of pairs $\cP$ as in Remark \ref{ourpairs} such that if $i \le j$, $p_i \leq p_j$. Define $\displaystyle \varinjlim_{j \in \Gamma} p_j = \bigvee_{j \in \Gamma} p_j$. 
 Note that for any pair $(L,M)$, we also have $\bigg(\displaystyle \varinjlim_{j \in \Gamma} p_j\bigg)(L,M) = \displaystyle \bigcup_{i \in \Gamma} p_j(L,M).$

For a poset $(P,\leq)$, its \emph{dual} $(P^\vee, \preccurlyeq)$ is the poset whose elements are the same as $P$ and such that $p \preccurlyeq q$ precisely when $q \leq p$.

Let $\Omega$ be an inverse poset (i.e., $\Omega^\vee$ is directed) and $\{p'_j\}_{j \in \Omega}$ a set of pair operations  such that if $i \leq j$, then $p'_i\leq p'_j$. Define $\displaystyle \varprojlim_{j \in \Omega} p'_j = \bigwedge_{j \in \Omega} p'_j$. Thus, for any pair $(L,M)$, $\bigg(\displaystyle \varprojlim_{j \in \Omega} p'_j \bigg)(L,M)=\displaystyle \bigcap_{i \in \Gamma} p_j(L,M)$.
\end{defn}

The following result collects some generalizations of \cite[Propositions 7.2 and 7.4]{nmeRG-cidual} from limits of submodule selectors to meets and joins 
of pair operations, along with many additional easily-proved results.

 \begin{prop} \label{pr:meetjoin}
     Let $\Sigma$ be a nonempty indexed set. 
 Let $\{p_j \mid j \in \Sigma\}$ and $\{p_j' \mid j \in \Sigma\}$ be indexed sets of pair operations.  Let $p = \displaystyle \bigvee_{j \in \Sigma} p_j$ and $p' = \displaystyle \bigwedge_{j \in \Sigma} p_j'$.  
 Consider the following properties \small{$\mathscr{P}$}:
 \vspace{-10pt}
 \begin{multicols}{2}
     \begin{enumerate}
         \item \small{$\mathscr{P}=$} order preserving on submodules.
         \item \small{$\mathscr{P}=$} order preserving on ambient modules.
         \item \small{$\mathscr{P}=$} surjection-functorial.
         \item \small{$\mathscr{P}=$} functorial.
         \item \small{$\mathscr{P}=$} restrictable.
         \item \small{$\mathscr{P}=$} surjection-cofunctorial.
         \item \small{$\mathscr{P}=$} cofunctorial.  
         \item \small{$\mathscr{P}=$} absolute.
         \item \small{$\mathscr{P}=$} residual.
         \item \small{$\mathscr{P}=$} extensive.
         \item  \small{$\mathscr{P}=$} intensive.
         \item  \small{$\mathscr{P}=$} hereditary.
         \item \small{$\mathscr{P}=$} closure operation.
         \item \small{$\mathscr{P}=$} cohereditary.
         \item \small{$\mathscr{P}=$} interior operation.
     \end{enumerate}
     \end{multicols}
     If all of the $p_j$ (resp. $p'_j$)  satisfy any property \small{$\mathscr{P}$} in (1)--(11) then $p$ (resp. $p'$) also satisfies \small{$\mathscr{P}$}.

     If all of the  $p'_j$  satisfy property \small{$\mathscr{P}$} in (12) (resp. (13)) then so does $p'$.

     If all of the $p_j$ satisfy property \small{$\mathscr{P}$} in (14) (resp. (15)) then so does $p$.
 \end{prop}

  \begin{proof}
    The proofs of (1)-(2), (8), and (10)-(11) are immediate.

     (3) or (4) (join) Suppose all the $p_i$ are (surjection-)functorial, let $L \subseteq M$ be a pair and $g: M \ra M'$ a (surjective) map.  
     Let $x\in (\bigvee_j p_j)(L,M)$.  Then there exist $j_1, \ldots, j_n \in \Sigma$ and $x_i \in p_{j_i}(L,M)$ with $x = \sum_{i=1}^n x_i$.  Then for all $i$, $g(x_i) \in p_{j_i}(g(L),M')$, since the $p_{j_i}$ are  (surjection-)functorial.  Thus, $g(x) = g(\sum_i x_i) = \sum_i g(x_i) \in \left(\bigvee_j p_j\right)(g(L),M')$. 

     (3) or (4) (meet) Let $g$ be as above and let $x\in p'(L,M)$.  Then for all $j\in \Sigma$, $x \in p_j'(L,M)$.  It follows that for all $j$, $g(x) \in p_j'(g(L),M')$.  That is, $g(x) \in \bigcap_j p_j'(g(L),M') = p'(g(L),M')$.

       (5) (join) Let $L, K \subseteq M$. Then \[ p(L\cap K, K) = \sum_{j\in \Sigma} p_j(L\cap K, K) \subseteq \sum_{j \in \Sigma} p_j(L,M) = p(L,M).\]

     (5) (meet) 
     Let $L,K,M$ be as above.  Then \[p'(L\cap K, K)= \bigcap_{j\in \Sigma} p'_j(L\cap K, K) \subseteq \bigcap_{j\in \Sigma} p'_j(L,M)=p'(L,M).\]  

     (6) or (7) (join) Suppose that the $p_j$ are all (surjection-)cofunctorial, 
     $g:M \rightarrow M'$ is a (surjective) $R$-module homomorphism, and $L \subseteq M'$.  Suppose $x \in p(g^{-1}(L),M)$.  Then there exist $j_1, \ldots,j_n \in \Sigma$ and $x_i \in p_{j_i}(g^{-1}(L), M)$ such that $x =\sum_{i=1}^n x_i$.  As each $p_{j_i}$ is (surjection-)cofunctorial, $ p_{j_i}(g^{-1}(L),M) \subseteq g^{-1}(p_{j_i}(L,M'))$.  Thus, $g(x_i) \in p_{j_i}(L,M')$, so $g(x) = \sum_i g(x_i) \in \sum_i p_{j_i}(L,M') \subseteq (\bigvee_j p_j)(L,M') = p(L,M')$, which in turn implies that $x \in g^{-1}(p(L,M'))$. Thus $p$ is (surjection-)cofunctorial. 

     (6) or (7) (meet) Let $L,M,M',g$ be as above.  We have \[
     p'(g^{-1}(L),M) = \bigcap_{j \in \Sigma} p_j'(g^{-1}(L),M) \subseteq \bigcap_{j \in \Sigma} g^{-1}(p_j'(L,M')) = g^{-1}\Big(\bigcap_{j \in \Sigma} p_j'(L,M')\Big) = g^{-1}(p(L,M')).
     \]
     

     (9) (join) Let $L \subseteq N \subseteq M$ and let $\pi: M \onto M/L$ be the natural map.  First note that for any $j \in \Sigma$, we have $L = \ker\pi \subseteq \pi^{-1}(p_j(N/L, M/L)) = p_j(N,M)$.  Hence, $L \subseteq p(N,M)$, since $p(N,M)$ contains all the $p_j(N,M)$.

     First suppose $x\in p(N,M)$.  Then there exist $j_1, \ldots, j_n \in \Sigma$ and $x_i \in p_{j_i}(N,M)$ for each $i$, such that $x=\sum_i x_i$.  But then $\pi(x_i) \in p_{j_i}(N/L,M/L)$, so that $\pi(x) = \sum_{i=1}^n \pi(x_i) \in p(N/L,M/L)$.  Hence, $x \in \pi^{-1}(p(N/L,M/L))$.
    
    Conversely suppose $x \in \pi^{-1}(p(N/L,M/L))$.  Then there exist $j_1, \ldots, j_n \in \Sigma$ and $y_i \in p_{j_i}(N/L,M/L)$ with $\pi(x) = \sum_i y_i$.  For each $i$, choose $x_i \in M$ with $\pi(x_i)=y_i$.  Then $x_i \in \pi^{-1}(p_{j_i}(N/L,M/L)) = p_{j_i}(N,M)$, so that $\sum_{i=1}^n x_i \in p(N,M)$.  Since $\pi(x) = \pi(\sum_{i=1}^n x_i)$, we have $x-\sum_{i=1}^n x_i \in \ker\pi = L \subseteq p(N,M)$.  It follows that $x\in p(N,M)$.

     (9) (meet) We have \begin{align*}
     x \in p'(N,M) &\iff \forall j\in \Sigma \text{, } x \in p_j'(N,M) = \pi^{-1}(p_j'(N/L, M/L)) \\
     &\iff \forall j \in \Sigma \text{, } \pi(x) \in p_j'(N/L,M/L) \iff \pi(x) \in p'(N/L,M/L) \\
     &\iff x \in \pi^{-1}(p'(N/L,M/L))
     \end{align*}

     (12) Let $L \subseteq N \subseteq M$ 
     . Then: \[
     p'(L,M)\cap N = \left(\bigcap_{j \in \Sigma} p_j'(L,M)\right) \cap N = \bigcap_{j \in \Sigma} (p_j'(L,M) \cap N) = \bigcap_{j \in \Sigma} p_j'(L,N) = p'(L,N).
     \]

     (13) Suppose the $p'_j$ are closure operations for all $j$ and $L \subseteq M$.  Since  $p' = \displaystyle \bigwedge_{j \in \Gamma} p_j'$, then by parts (1) and (10), $p'$ is order preserving on submodules and extensive.  Thus, we need only show that $p'$ is idempotent.  By extensivity, we have that $L \subseteq p'(L,M)$, and then we obtain that $p'(L,M) \subseteq p'(p'(L,M),M)$ since $p'$ is order preserving on submodules.  Suppose $x \in p'(p'(L,M),M)$, then $x \in p'_j(p'(L,M),M)$ for all $j$, so that since $p' \leq p'_j$ and $p'_j$ is order-preserving on submodules, we have  $x \in p'_j(p'_j(L,M),M)=p'_j(L,M)$ for all $j$ since $p'_j$ is idempotent.  Hence, $x \in p'(L,M)$.

     (14) Let $L \subseteq N \subseteq M$, and $x\in M$.  Suppose $x+L \in p(N/L,M/L)$.  Then there exist $j_1, \ldots, j_n \in \Sigma$ and $x_1, \ldots, x_n \in M$ such that $x+L = \sum_{i=1}^n x_i + L$ in $M/L$, and such that  for each $1\leq i \leq n$, $x_i + L \in p_{j_i}(N/L,M/L) = \frac{p_{j_i}(N,M)+L}L$ since $p_{j_i}$ is cohereditary.  Thus, \[
     x+L \in \frac{\left(\sum_i p_{j_i}(N,M)\right)+L}L \subseteq \frac{p(N,M)+L}L.
     \]
     Conversely, suppose $x+L \in \frac{p(N,M)+L}L$.  Then there is some $\ell \in L$ with $x+\ell \in p(N,M)$, whence there exist $j_1, \ldots, j_n \in \Sigma$ and $x_i \in p_{j_i}(N,M)$ with $x+\ell = \sum_{i=1}^n x_i$.  But since $x_i \in p_{j_i}(N,M)$, we have $x_i + L \in \frac{p_{j_i}(N,M)+L}L = p_{j_i}(N/L, M/L)$ since $p_{j_i}$ is cohereditary.  Thus, $x+L = \sum_i x_i + L \in p(N/L, M/L)$.

     (15) Note that $p$ is order-preserving on submodules by (1) and intensive by (11), so we need only show it is idempotent.  By intensivity we have $p(L,M) \subseteq L$, so that by the order-preservation property we have $p(p(L,M),M) \subseteq p(L,M)$.  Accordingly, we need only show that $p(L,M)  \subseteq p(p(L,M),M)$.  So let $x \in p(L,M)$.  Then there exist $j_1, \ldots, j_n \in \Sigma$ with $x_i \in p_{j_i}(L,M)$ for each $i$ and $x = \sum_{i=1}^n x_i$.  But since $p_{j_i}$ is idempotent and order-preserving on submodules and since $p_{j_i} \leq p$, we have \[
     x_i \in p_{j_i}(L,M) = p_{j_i}(p_{j_i}(L,M),M) \subseteq p_{j_i}(p(L,M),M).
     \]
     Thus, $x = \sum_i x_i \in p(p(L,M), M)$, completing the proof that $p$ is an interior operation.
 \end{proof}

\begin{example}
Although 
meets of closure operations are closure operations, the same is not true of joins. 
 For example, let $R = k[x,y]$ be the polynomial ring in two variables $x,y$ over a field $k$, and let $\cl$ and $\cl'$ be the  operations given by \[
I^\cl = \begin{cases}
    (x) &\text{if } I \subseteq (x),\\
    R &\text{otherwise,}
\end{cases} \quad \text{and }
I^{\cl'} := \begin{cases}
    (y) &\text{if } I \subseteq (y),\\
    R &\text{otherwise.}
    \end{cases}
\]
It is immediate that $\cl$ and $\cl'$ are closure operations on the ideals of $R$.  Let $\rp = \cl \vee \cl'$.  Then $(0)^\rp = (x,y)$, but $(x,y)^\rp = R$, whence $\rp$ is not idempotent and thus not a closure operation.

Similarly, although joins of interior operations are interior operations, the same is not true of meets. Indeed, if we define interior operations $\intr$ and $\intr'$ on the ideals of $R$ by \[
I_\intr = \begin{cases}
    (x) &\text{if } x\in I,\\
    0 &\text{otherwise,}
    \end{cases} \quad \text{and } I_{\intr'} = \begin{cases}
    (y) &\text{if } y\in I,\\
    0 &\text{otherwise.}
    \end{cases}
\]
Let $\riq = \intr \wedge \intr'$.  Then $R_\riq = (xy)$, but $(xy)_\riq = 0$, so $\riq$ is not idempotent and thus not an interior operation. Thus, meets of interior operations need not be interior operations.
\end{example}

\begin{example} Although the cohereditary property is closed under joins, the same is not true for meets, even when they are inverse limits.  In particular, let $R = k[x,y]$, and for all positive integers $j$, let $p_j'$ be defined on pairs of $R$-modules by $p_j'(N,M) = x^jN$.  This is a cohereditary (and absolute) operation. Let $p' = \displaystyle \varprojlim_j p_j'$, so that $p'(N,M) = \bigcap_j x^j N$.  
 Let $I = (xy-y)R$.  Then $\frac{p'(R,R)+I}I = \frac{(\bigcap_j x^jR) +I}I = I/I=0$, since $\bigcap_j x^jR=0$. 
 But for all $j\geq 1$, we have $y-x^jy = (\sum_{i=0}^{j-1} x^i)(y-xy)\in I$, so $y \in (x^jR) + I$ for all $j$.  Hence, $y+I \in\frac{\bigcap_j ((x^jR)+I)}I  = \bigcap_j x^j (R/I)=p'(R/I,R/I)$.  Since $y\notin I$, $p'$ is not cohereditary.

\end{example}

\begin{example}
Although the hereditary property is closed under meets, the same is not true for joins.
For example, let $R = k[x,y, z,w] /(xz,yw)$, where $k$ is a field and $x,y,z,w$ are indeterminates in the ambient polynomial ring 
Let $p,q$ be the pair operations on pairs of $R$-modules given by $p(L,M) := (L:_Mx)$ and $q(L,M) := (L:_My)$. It is clear that $p$ and $q$ are hereditary.  However, $p \vee q$ is not.

Indeed, let $J = (z+w)R$.  Then $p(0,R) = (0:_Rx) = zR$ and $q(0,R) = (0:_Ry) = wR$, so that $J \subseteq p(0,R) + q(0,R) = (p\vee q)(0,R)$ since  $z+w \in zR+wR$, whence $(p \vee q)(0,R) \cap J = J =  (z+w)R$.  On the other hand, $p(0,J) = (0:_Jx) =  zR \cap J = (z^2 + zw, yz)$ and $q(0,J) = (0:_Jy)= wR \cap J = (zw+w^2, xw)$, so that $(p \vee q)(0,J) =  (z^2 + zw, yz, zw + w^2, xw)$, which is a proper subset of $J$.  That is, $(p \vee q)(0,R) \cap J \neq (p \vee q)(0,J)$, so $p\vee q$ is not hereditary, even for pairs consisting of ideals.
\end{example}

Even so, the hereditary property is closed under \emph{direct limits}, as the next result shows.

 \begin{prop} \label{pr:heredlim} Let $\Gamma$ be a directed poset
 . Let $\{p_j \mid j \in \Gamma\}$ be pair operations such that $p_i \leq p_j$ whenever $i \leq j$
 .  Let
 $p=\displaystyle\varinjlim_{j \in \Gamma} p_j$
 .
 If all of the $p_j$ 
 are hereditary, then so is $p$.
 \end{prop}

\begin{proof}
  %
     Let $L \subseteq N \subseteq M$. 
 Then \[
 p(L,N) = \bigcup_{j \in \Gamma} p_j(L,N) = \bigcup_{j \in \Gamma} \left(p_j(L,M) \cap N\right) = \Big(\bigcup_{j \in \Gamma} p_j(L,M)\Big) \cap N = p(L,M) \cap N.  \qedhere
 \] 
\end{proof}




Next, we interface with our notion of duality. 


\begin{prop}\label{pr:limduals}
    Let $R$ and $\cP$ be as in Remark~\ref{rem:matlisdualizable}. 
 Let $\Sigma$ be an indexed set and let $\{p_j \mid j \in \Sigma\}$ be pair operations on $\cP$.  Then \[\displaystyle \Big(\bigvee_{j \in \Sigma} p_j\Big)^\dual = \displaystyle \bigwedge_{j \in \Sigma} (p_j^\dual),\] and \[\displaystyle\Big(\bigwedge_{j \in \Sigma} p_j\Big)^\dual=\bigvee_{j \in \Sigma} (p_j^\dual).\]
\end{prop}

\begin{proof}
     The proof is similar to \cite[Propositions  7.7, and 7.8]{nmeRG-cidual} but we include the proof for completeness as the joins here are sums and not unions. 
    
Let $z\in \Big(\bigvee\limits_{j \in \Sigma} p_j\Big)^\dual (L,M)
$. By Proposition \ref{pr:kernelview}, we have $g(z)=0$ for all $g \in \Big(\bigvee\limits_{j \in \Sigma}p_j\Big)((M/L)^\vee, M^\vee) = \sum\limits_{j \in \Sigma} p_j((M/L)^\vee,M^\vee) $
.  In particular, for all $j \in \Sigma$ and  $g \in p_j((M/L)^\vee,M^\vee) $, $g(z)=0$.  Hence, $z \in \bigwedge\limits_{j \in \Sigma} p_j^\dual(L,M)$.
    
    Conversely, let $z\in \bigwedge\limits_{j \in \Sigma} (p_j^\dual(L,M))=\bigcap\limits_{j\in \Sigma} (p_j^\dual(L,M))$.  Then for all $j \in \Sigma$ and all $g \in p_j((M/L)^\vee, M^\vee)$, we have $g(z) = 0$ by Proposition \ref{pr:kernelview}.
Now let $h \in (\bigvee_j  p_j)((M/L)^\vee, M^\vee)$.  Then $h = h_1 + \cdots + h_n$, where each $h_i \in p_{j_i}((M/L)^\vee, M^\vee)$ for some $j_1,\ldots, j_n \in \Sigma$.  In particular, each $h_i(z) = 0$, so $h(z)=0$.  Since $h$ was arbitrary, $z\in \Big(\bigvee\limits_{j\in \Sigma} p_j\Big)^\dual (L,M)$.

Now if we set $p_j=q_j^\dual$ in the expressions above, we obtain, $\Big(\bigvee\limits_{j \in \Sigma} q_j^\dual\Big)^\dual = \displaystyle \bigwedge_{j \in \Sigma} (q_j^\dual)^\dual$ and after applying duals on both sides and using Lemma \ref{lem:doubledual}, we obtain $\bigvee\limits_{j \in \Sigma} q_j^\dual=\Big(\displaystyle \bigwedge_{j \in \Sigma} q_j\Big)^\dual$ which gives the second expression.
\end{proof}

Although the cohereditary property is not preserved by inverse limits for a general ring, by our duality we obtain:

\begin{cor}
    Let $(R,\m)$ a complete Noetherian local ring. Let $\Omega$ be an inverse poset.  Let $\{p'_j \mid j \in \Omega\}$ be pair operations defined on Matlis dualizable $R$-modules such that $p'_i \leq p'_j$ whenever $i \leq j$.  Let $p'=\displaystyle\varprojlim_{j \in \Omega} p'_j$.  If all of the $p'_j$ are cohereditary, then so is $p'$.
\end{cor}

\begin{proof}
  Note that for all $j$, $(p'_j)^\dual$ are hereditary (refer to Table \ref{tab:dual}). By Proposition \ref{pr:heredlim}, $\displaystyle\varinjlim_{j \in \Omega^\vee} (p'_j)^\dual$ is hereditary.  By Proposition \ref{pr:limduals}, \[\displaystyle\varinjlim_{j \in \Omega^\vee} (p'_j)^\dual=\bigvee_{j \in \Omega^\vee} (p'_j)^\dual=\Big(\bigwedge_{j \in \Omega} p'_j\Big)^\dual=\bigg(\displaystyle\varprojlim_{j \in \Omega} p'_j\bigg)^\dual=(p')^\dual.\] Again we make use of Table \ref{tab:dual} and Lemma \ref{lem:doubledual} to see that $p'=((p')^\dual)^\dual$ is cohereditary. 
\end{proof}

The following is an example of an inverse limit of interior operations that is not itself an interior operation.

\begin{example}
Let $k$ be a field and $R=k[[x,y]]/(xy)$.  Note that the nonzero and non-unit ideals have the form 
\begin{align*}
    &(x^n,y^m) \text{ for } n,m \geq 1, \\
    &(x^n+ay^m) \text{ for } n,m \geq 1 \text{ and } a \in k^\times,\\
    &(x^n) \text{ for } n \geq 1\\
    &(y^m) \text{ for } m \geq 1.
\end{align*}

Define the interior operations $i_j$ on the ideals of $R$ as follows:

\[I_{i_j}^R=\begin{cases}
    I \text{ if } I=R \text{ or } I=(x^n,y^m) \text{ for } m \geq j \text{ or } I=(x^n+ay^m) \text{ for } m \geq j \text{ and } a \in k^\times\\
    (x^n,y^j) \text{ if } I=(x^n,y^m) \text{ for } 1 \leq m \leq j  \text{ or }  I=(x^{n-1}+ay^{m-1}) \text{ for }  2 \leq n, 2 \leq  m \leq j \text{ and } a \in k^\times\\
    0 \text{ if } I=0, I=(x^n) \text{ for } n\geq 1 \text{ and } I=(y^m) \text{ for } m\geq 1
\end{cases}\]

Note that $i_j \geq i_\ell$ for $j \leq \ell$ since $(x^n,y^j) \supseteq (x^n,y^{\ell})$ if $j \leq \ell$.
Set $i=\displaystyle \varprojlim_{j \in \mathbb{N}} i_j$.

Note that $(x,y)^R_i=(x)$, but $(x)^R_i=0$.  Thus $i$ is not idempotent, thus not an interior operation.

\end{example}


\section{Applications to known pair operations}
This section serves to connect the above theory with pair operations that appear in the literature, especially some common closure and interior operations.

\subsection{A table of operations and their properties}\label{sec:bigtable}

In Table~\ref{ta:exprops}, we link various pair operations to properties they may or may not have.  This should help the reader's intuition both for the operations and the properties.  Some of the operations in the table are not defined in this paper.  However, for all such operations, the references in the table will point the reader to good places to learn about them.

\begin{rem}\label{rem:taprops}
Table~\ref{ta:exprops} is largely self-explanatory; however, some clarifying comments are in order. \begin{itemize}
    \item The `scope' column means, ``if the operation is taken with the given scope, then the properties it does or does not satisfy are as given in this row of the table".  Some operations on the table have bigger scope as defined in the literature.  For instance, tight closure is defined over arbitrary Noetherian rings containing a field \cite{HHmain}, but if there is no weak test element, it is not clear that the operation is idempotent on infinitely generated modules.  The scopes given are meant to encapsulate the situations where most researchers would use or care about the operations in question.
    \item We have tried to strike a balance between completeness and not cluttering the table overly much.  Thus, no citation is given for items that are easily determined from the definition.  For example, it is easy to see from the definition of (subset) module closures that they are residual.
    \item By the same token, some implications hold automatically.  For instance, no nontrivial extensive operation can be absolute or intensive, and no nontrivial intensive operation can be residual or extensive. Moreover, by Proposition~\ref{pr:func}, any pair operation that is \opsub\ is functorial if and only if it is cofunctorial.
    \item Also, some of the operations in the list are special cases of other operations. For example, Frobenius and plus closures are special cases of module closures; thus anything that holds for all module closures holds in particular for Frobenius and plus closures.  Similarly, any module closure is a subset module closure.
    \item Persistence holds for tight closure when $R \ra S$ are in the given scope \emph{and} $R$ is either F-finite or essentially of finite type over an excellent local ring \cite{HHbase}.
    \item Persistence for $J$-basically full closure means: if $L \subseteq M$ are $R$-modules, $J$ an ideal of $R$, and $R \ra S$ a ring map, then if $x\in \Jcol J L M$, then $x \otimes 1 \in M \otimes_R S$ is in $\Jcol {{(JS)}} {(LS)} {M \otimes_R S}$.  Thus, though basically full closure over a local ring $(R,\m)$ is the same as $\m$-basically full closure, \emph{persistence} works differently, as in a local homomorphism $(R,\m) \ra (S,\n)$, typically $\m S \neq \n$.
    \item Persistence along a ring map $R \ra S$ for a subset module closure $\cl_{X,L}$ (resp. subset module trace $\tr_{X,L}$) means that it persists to $\cl_{XS, L \otimes_R S}$ (resp. $\tr_{XS, L\otimes_R S}$), where $XS = \{ x\otimes s \mid x\in X,\ s \in S\}$.
    \item Persistence for $\ia$-saturation means that $(L :_M \ia^\infty)S \subseteq (LS :_{M \otimes S} (\ia S)^\infty)$.
    \item When asking whether a closure (resp. interior) operation is Nakayama, we are implicitly restricting the scope of the modules in question to be Noetherian (resp. Artinian) and the rings $(R,\m)$ to be Noetherian local.  Moreover, plus closure is only known to be Nakayama when the ring is complete, and module closures $\cl_L$ are known to be Nakayama only when the module $L$ is finitely generated or is a local $R$-algebra whose maximal ideal contracts to $\m$, and certain  other similar cases.  We only know module trace $\tr_L$ to be a Nakayama interior when $L$ is finitely generated.  Moreover, there is a version of the Nakayama property for certain kinds of extensive idempotent operations that aren't necessarily closures; the Ratliff-Rush operation is an example.  Hence all the asterisks in the Nakayama column.
\end{itemize}
\end{rem}

{
\begin{table}
\centering
\resizebox{\columnwidth}{!}{%
\begin{NiceTabular}{c||c|c|c|c|c|c|c|c|c|c|c|c|c|c|c|}
\ & Scope & {\rotate \Block{}{order-preserving\\(on submodules)}} & {\rotate idempotent} & {\rotate extensive} & {\rotate closure operation}& {\rotate intensive} & {\rotate interior operation} & {\rotate residual} & {\rotate absolute}& {\rotate Nakayama$^*$} & {\rotate hereditary}& {\rotate cohereditary} & {\rotate functorial} & {\rotate cofunctorial} & {\rotate persistent$^*$}\\
\hline \hline
identity
 & modules & \cmark & \cmark & \cmark & \cmark & \cmark & \cmark &  \cmark & \cmark & \cmark & \cmark & \cmark & \cmark & \cmark & \cmark \\
\hline
indiscrete closure
 & modules & \cmark & \cmark & \cmark & \cmark & \xmark & \xmark & \cmark & \xmark & \cmark & \cmark & \cmark & \cmark & \cmark & \cmark \\
\hline
 zero 
 & modules & \cmark & \cmark & \xmark & \xmark & \cmark & \cmark & \xmark & \cmark & \cmark & \cmark & \cmark & \cmark & \cmark & \cmark \\
\hline
radical & single ideals & \cmark & \cmark & \cmark & \Block{}{\cmark\\ [\ref{guide}]2.1.2} & \xmark & \xmark & - & - & \Block{}{\xmark\\ [\ref{guide}]5.1} & - & - & \Block{}{\cmark \\ \ref{pr:radfunctcofunct}} & \Block{}{\cmark \\ \ref{pr:radfunctcofunct}} & \Block{}{\cmark \\ \ref{ex:persprops}} \\
\hline
\Block{}{tight closure
} &\Block{}{modules over prime \\ characteristic  Noetherian  \\rings admitting \\ big weak test elements}& \cmark & \cmark & \cmark & \Block{}{\cmark\\ [\ref{tcop}]8.4} & \xmark & \xmark & \cmark & \xmark & \Block{}{\cmark\\ [\ref{phandep}]3.2} & \Block{}{\xmark \\ \ref{ex:tcnother}} &\cmark & \Block{}{\cmark\\ [\ref{tcop}]5.3} & \cmark & \Block{}{\cmark$^*$ \\ [\ref{tcop}]8.6} \\
\hline
$\ia$-tight closure &\Block{}{modules over prime \\ characteristic Noetherian \\ rings admitting \\ $\ia$-test elements} & \Block{}{\cmark \\ [\ref{HaYo}]1.3} & \Block{}{\xmark\\ [\ref{HaYo}]1.4} & \Block{}{\cmark \\ [\ref{HaYo}]1.3} &\xmark &\xmark &\xmark & \Block{}{\cmark \\ [\ref{HaYo}]1.3} &\xmark &- &\xmark&\cmark &\Block{}{\cmark \\ \ref{pr:atcfunc}}  &\cmark &?\\ 
\hline
Frobenius closure & \Block{}{modules over\\prime characteristic \\ Noetherian  rings}& \cmark & \cmark & \cmark & \Block{}{\cmark \\ [\ref{tcop}]3.5} & \xmark & \xmark & \cmark & \xmark & \Block{}{\cmark\\ [\ref{chdual}]4.8} & \Block{}{\xmark \\ \ref{ex:tcnother}} & \cmark & \Block{}{\cmark\\ [\ref{tcop}]3.5} & \cmark & \Block{}{\cmark \\ \ref{ex:persprops}} \\
\hline
plus closure & \Block{}{modules over  domains\\  or over 
Noetherian
rings} & \cmark & \cmark & \cmark & \cmark & \xmark & \xmark & \cmark & \xmark & \Block{}{\cmark\\ [\ref{chdual}]4.8*} & \Block{}{\xmark \\ \ref{ex:tcnother}} & \cmark & \cmark & \cmark & \Block{}{\cmark\\ [\ref{guide}]4.3} \\
\hline
module closures & modules & \cmark & \cmark & \cmark & \Block{}{\cmark \\ [\ref{guide}]7.0.6} & \xmark & \xmark & \cmark & \xmark & \Block{}{\cmark\\ [\ref{chdual}]4.1,7*} & \xmark & \cmark & \cmark & \cmark & \cmark \\
\hline
\Block{}{subset module\\ closure $\cl_{X,L}$} & modules & \cmark & \Block{}{\xmark\\ \ref{rem:smcnotidem}} & \cmark & \xmark & \xmark & \xmark & \Block{}{\cmark \\ \ref{pr:smc}} & \xmark & - & \xmark & \cmark & \Block{}{\cmark \\ \ref{pr:smc}} & \cmark & \Block{}{\cmark \\ \ref{pr:smc}} \\
\hline
\Block{}{socle $(L :_M \m)$} & \Block{}{modules over \\ local rings $(R,\m)$} &\cmark &\xmark &\cmark &\xmark &\xmark &\xmark &\cmark & \xmark &- &\cmark &\cmark &\cmark & \cmark &\Block{}{\xmark \\ \ref{ex:persprops}} \\
\hline
\Block{}{$\ia$-saturation \\ $(L :_M \ia^\infty)$} & \Block{}{modules over \\ Noetherian  rings} & \cmark & \Block{}{\cmark\\ [\ref{lcbook}]2.1.2}  & \cmark & \cmark & \xmark & \xmark & \cmark & \xmark & \xmark & \Block{}{\cmark \\ [\ref{lcbook}]1.1.6} & 
\cmark & \Block{}{\cmark \\ [\ref{lcbook}]1.1.1} & \cmark & \Block{}{\cmark \\ [\ref{lcbook}]4.2.1} \\
\hline
\Block{}{integral closure\\ (EHU) \cite{EHU-Ralg}} & \Block{}{finitely generated modules \\ over Noetherian rings} & \cmark & \cmark & \cmark & \cmark & \xmark & \xmark & \xmark & \xmark & \Block{}{\xmark\\ [\ref{lint}]rmk\\ p.9} & \Block{}{\cmark\\ [\ref{howext}]9.6} & \xmark & \Block{}{\cmark\\ [\ref{howext}]9.6} & \cmark & \cmark \\
\hline
\Block{}{integral closure \\ (Rees) \cite{Rees-redmod}} & \Block{}{modules over \\Noetherian rings} & \cmark & \cmark & \cmark & \cmark & \xmark & \xmark & \xmark & \xmark & \Block{}{\xmark\\ [\ref{lint}]rmk\\ p.9}& \Block{}{\xmark\\ [\ref{howext}]9.8} & \xmark& \Block{}{\cmark \\ [\ref{howext}]2.12} & \cmark & \cmark \\
\hline
\Block{}{integral closure \\ (liftable) \cite{nmeUlr-lint}} & modules & \cmark & \cmark & \cmark & \Block{}{\cmark \\ [\ref{lint}]2.4} & \xmark & \xmark & \cmark & \xmark & ? & \Block{}{\xmark \\ \ref{ex:licnonher}} & \cmark & \Block{}{\cmark \\ [\ref{lint}]2.4} & \cmark & \Block{}{\cmark \\ [\ref{lint}]2.4}\\
\hline
\Block{}{$J$-basically \\ full closure}& modules &  \cmark & \cmark & \cmark & \Block{}{\cmark\\ [\ref{nonres}]4.7} & \xmark & \xmark & \Block{}{\xmark
} & \Block{}{\xmark
} & \Block{}{\cmark \\ [\ref{nonres}]4.7} & \Block{}{\cmark \\ [\ref{howext}]7.3} & 
\xmark
& \Block{}{\cmark\\ [\ref{nonres}]4.7} & \cmark & \Block{}{\cmark \\ \ref{ex:persprops}} \\
\hline
\Block{}{basically full\\ closure \\ \cite{HRR-bf}} & \Block{}{pairs $(L,M)$ of \\
finite modules over \\ Noetherian local rings\\ with $M/L$ artinian} & \cmark & \cmark & \cmark & \Block{}{\cmark \\ [\ref{bf}]4.2
} & \xmark & \xmark & \Block{}{\xmark \\ [\ref{nonres}]4.8} & \xmark & \cmark & \cmark & \xmark & \cmark & \cmark & \Block{}{\xmark \\ [\ref{guide}]4.3} \\
\hline
\Block{}{tight interior\\ \cite{nmeSc-tint}} & \Block{}{modules over F-finite \\ Noetherian  rings} & \cmark & \Block{}{\cmark \\ [\ref{tint}]2.7} & \xmark & \xmark  & \cmark & \cmark & \xmark & \cmark & ? & \cmark & \Block{}{\xmark \\ \ref{ex:tcnother}} & \Block{}{\cmark \\ [\ref{tint}]2.1}  & \cmark & \Block{}{\xmark \\ \ref{ex:persprops}} \\ 
\hline
trace $\tr_L$ & modules & \cmark & \cmark  & \xmark & \xmark  & \cmark & \Block{}{\cmark\\ [\ref{cid}]8.4} & \xmark & \cmark & \Block{}{\cmark \\ \ref{pr:trNak}$^*$} & \cmark & \Block{}{\xmark \\ \ref{ex:tcnother}} & \Block{}{\cmark\\ [\ref{cid}]8.4} & \cmark & \cmark \\ 
\hline
\Block{}{subset module\\ trace $\tr_{X,L}$}& modules & \cmark & \Block{}{\xmark\\ [\ref{cid}]8.5} & 
\xmark & \xmark  & \cmark & \xmark & \xmark & \cmark & - & \cmark & \xmark & \Block{}{\cmark\\ [\ref{cid}]8.4} & \cmark & \Block{}{\cmark\\ [\ref{cid}]8.12} \\ 
\hline
\Block{}{$J$-basically \\ empty interior} & modules & \cmark & \cmark & \xmark & \xmark  & \cmark & \Block{}{\cmark\\ [\ref{nonres}]4.10} & \Block{}{ \xmark \\ \ref{ex:Jbenotabs}} & \Block{}{\xmark \\ [\ref{howext}]7.15
} & \Block{}{\cmark\\ [\ref{nonres}]4.10} & \Block{}{\xmark\\ [\ref{howext}]7.15
} & \Block{}{\cmark\\ [\ref{howext}]7.12} & \Block{}{\cmark\\ [\ref{nonres}]4.10} & \cmark & \Block{}{\cmark \\ \ref{ex:persprops}}\\
\hline
\Block{}{Ratliff-Rush\\ operation \cite{RatRu-rr}} & \Block{}{nested pairs of ideals \\
in integral domains} & \Block{}{\xmark\\ [\ref{coblo}]1.1} & \cmark & \cmark &\xmark & \xmark & \xmark &- & \xmark & \Block{}{\cmark \\ { [\ref{RRop}]3.2.3$^*$}} & \cmark &- &\Block{}{\cmark \\ [\ref{howext}]3.8}&\Block{}{\xmark \\ [\ref{nonres}]2.6} &\cmark \\
\hline
core (of ideals) & single ideals & \Block{}{\xmark \\ [\ref{Lcore}]1} & \Block{}{\xmark \\ \ref{ex:corenonid}} & \xmark & \xmark & \cmark & \xmark & - & \cmark & - & \cmark & - & \Block{}{\xmark \\ \ref{ex:corenotfunc}} & \Block{}{\xmark\\ \ref{ex:corenotcof}} &\Block{}{\xmark \\ [\ref{HTcore}]2.4}\\
\hline
\end{NiceTabular}%
}
\caption{A grid of pair operations and properties.  \\
In this table, \cmark\ means the property holds for the given operation, \xmark\ means there are counterexamples, - means the property does not apply to the operation, and ? means we do not know whether the property always holds for this operation given its scope.
\\
Below some \cmark\ and \xmark\ signs are references that indicate a property does or does not hold. Reference numbers: 
[\nextref{lcbook}\ref{lcbook}] is \cite{BrSh-locobook2};
[\nextref{phandep}\ref{phandep}] is \cite{nmepdep}; 
[\nextref{guide}\ref{guide}] is \cite{nme-guide2}; 
[\nextref{tcop}\ref{tcop}] is \cite{nme-itcop}; [\nextref{cid}\ref{cid}] is \cite{nmeRG-cidual}; [\nextref{nonres}\ref{nonres}] is  \cite{ERGV-nonres};
[\nextref{chdual}\ref{chdual}] is \cite{ERGV-chdual}; [\nextref{howext}\ref{howext}] is \cite{ERGV-extend}; 
[\nextref{tint}\ref{tint}] is \cite{nmeSc-tint}; 
[\nextref{lint}\ref{lint}] is \cite{nmeUlr-lint}; 
[\nextref{coblo}\ref{coblo}] is \cite{HJLS-coeff}; [\nextref{bf}\ref{bf}] is \cite{HRR-bf};
[\nextref{HTcore}\ref{HTcore}] is \cite{HT-core}; 
[\nextref{HaYo}\ref{HaYo}] is \cite{HaYo-atc}; 
[\nextref{Lcore}\ref{Lcore}] is \cite{Lee-core}; [\nextref{RRop}\ref{RRop}] is \cite{RRAsymprime}.
}
\label{ta:exprops}
\end{table}
}

%
%


We include two propositions and several examples below to justify some of the assertions in Table \ref{ta:exprops}.

\begin{prop} \label{pr:radfunctcofunct}
    Radical is 
    fully functorial on ideals.  
\end{prop}
\begin{proof}

    Since radical is order-preserving, by 
    Proposition~\ref{pr:func} we need only show cofunctoriality. 
    By Remark \ref{rem:funcforideals}, we need to show that $\sqrt{(I:_Rx)} \subseteq \sqrt{I}:_Rx$. Suppose that $y \in \sqrt{(I:_Rx)}$.  Then there exists $n \in \mathbb{N}$ so that $y^n \in I:_Rx$ or $xy^n \in I$.  As $I$ is an ideal, $x^ny^n=x^{n-1}xy^n \in I$ implying that $xy \in \sqrt{I}$ or $y \in \sqrt{I}:_Rx$ which is what we needed to show.



\end{proof}
\begin{prop}\label{pr:atcfunc}
The $\ia$-tight closure operation is fully functorial. \end{prop}

\begin{proof} Let $R$ be a commutative ring of prime characteristic $p>0$, let $\ia$ be an ideal, let $L \subseteq M$ and $L' \subseteq M'$ be nested pairs of $R$-modules, let $g: M \ra M'$ be an $R$-linear map such that $g(L) \subseteq L'$, and let $z\in L^{*\ia}_M$.  Then by definition, there is some $c\in R^\circ$ and some $q_0$ a power of $p$ such that for all powers $q\geq q_0$ of $p$, we have $c \ia^{q} z^q_M \subseteq L^{[q]}_M$. 
But then $c\ia^{q} g(z)^{[q]}_{M'} \subseteq g(L)^{[q]}_{M'} \subseteq (L')^{[q]}_{M'}$ for all $q\geq q_0$, whence $g(z) \in (L')^{*\ia}_{M'}$.
\end{proof}

\begin{example} \label{ex:corenotcof}


Core is not cofunctorial.  Let $R=k[x,y]_{(x,y)} / (xy, y^2)$.  Note that $(x+ay)$ is a reduction of $(x,y)$ for any $a\in k$, since $(x+ay)(x,y) = (x^2) = (x,y)^2$.  Moreover, for any $a\neq b$ in $k$, we have $(x+ay) \cap (x+by) = (x^2)$, so that core$(x,y) \subseteq (x^2)$.  We have $((x,y) :x)=R$ since $x\in (x,y)$, so that core$((x,y):x) = R$ as well; the unit ideal is always basic.  But core$(x,y) : x \subseteq (x^2) :x =(x,y)$.  Thus, core$((x,y):x) \nsubseteq ($core$(x,y)) :x$, so that by Remark~\ref{rem:funcforideals}, core is not cofunctorial. 
\end{example}

\begin{example}\label{ex:corenotfunc}
Similarly, core is not functorial. Let $R=k[x,y,z]_{(x,y,z)} / (xyz^2, y^2 z^2)$. Let $I=(x,y)$ and $J=(x)$.  For any $n\in \N$, $I^n$ is minimally generated by the $n+1$ monomials $x^i y^{n-i}$, $0 \leq i \leq n$.  Indeed, if any were redundant, then since all relations in $R$ are monomial with respect to the ambient localized polynomial ring, such a redundant monomial $x^i y^{n-i}$ would have to be zero, but none are.  Then by \cite[Theorem 2, p. 151]{NR} and Nakayama's lemma, $I$ has no proper minimal reductions, so that core$(I)=I$.  However, $(zJ)(zI) = (x^2 z^2) = (zI)^2$, so that since $zJ \subseteq zI$, we have that $zJ$ is a reduction of $zI$.  In particular, core$(zI) \subseteq zJ = (xz)$.  But $yz \notin (xz)$, so that  $yz \in z$core$(I) \setminus $core$(zI)$.  Another appeal to Remark~\ref{rem:funcforideals} completes the proof. 
\end{example}

\begin{example}\label{ex:Jbenotabs}
%
The operation $J{\rm be}$ is not residual.
Let $R=k[\![x,y]\!]$, $L=\m^3$, $J=M=N=\m^2$ and let $\pi:M \rightarrow M/L$ be the canonical epimorphism.  Note that $N^M_{J{\rm be}}=\m^2(\m^2:_{\m^2}\m^2)=\m^4$.  However $$\pi^{-1}((N/L)^{M/L}_{J{\rm be}})=\pi^{-1}(\m^2(\m^2/\m^3:_{\m^2/\m^3} \m^2))=\pi^{-1}(\m^2(\m^2/\m^3))=\pi^{-1}(0)=\m^3 \neq \m^4.$$

\end{example}

\begin{example}\label{ex:licnonher}
    Liftable integral closure is not hereditary.  Let $(R,\m)$ be a Noetherian local domain and $J \subsetneq I$ be distinct ideals of $R$ such that $I$ is the integral closure of $J$ (e.g. $R=k[\![x,y]\!]$, $J=(x^2, y^2)$, $I = (x,y)^2$).  Then by \cite[Proposition 2.4(6)]{nmeUlr-lint}, $\lic J R = I$, so that $\lic J R \cap I = I$, but by \cite[Proposition 2.4(11)]{nmeUlr-lint}, $\lic JI \subseteq J + \m I$, which by Nakayama's lemma cannot equal $I$.  Thus, $\lic JI \neq \lic JR \cap I$. 
\end{example}

\begin{example}\label{ex:corenonid}
    Core is not idempotent. Let $S$ be a normal Cohen-Macaulay standard $\N$-graded domain over a field $k$ of characteristic zero, with $d=\dim S\geq 2$. Let $\m$ be its homogeneous maximal ideal.  Then by \cite[Corollary 6.4.4]{HySmi-Kawcore}, for all $n>0$, core$(\m^n) = \m^{nd + a+1}$, where $a$ is the a-invariant of $S$.  Then for any $n\geq -\frac a{d-1}$, we have $nd+a+1>n$, and hence $(nd+a+1)d+a+1 > nd+a+1$ as well.  Thus, $\cre(\cre(\m^n)) = \cre(\m^{nd+a+1}) = \m^{(nd+a+1)d+a+1} \neq \m^{nd+a+1} = \cre(\m^n)$.
\end{example}

\begin{example}\label{ex:tcnother}
    None of tight, Frobenius, nor plus closure are hereditary.  Let $(R,\m,k)$ be a prime characteristic complete F-finite non-F-pure Noetherian local domain, and let $I$ be an $\m$-primary ideal that is not Frobenius closed. Such an ideal exists because if $J$ is any ideal that isn't Frobenius closed, then there is some $n$ such that $J+\m^n$ is not Frobenius closed, since $J$ is the intersection of all the $J+\m^n$.  
    Let $u \in I^{\rm F}$ such that $(I :u) = \m$.  Let $J=(I,u)$.  Then $J/I \cong R/(I:u) = R/\m$, so by residuality, $I^*_J/I \cong \m^*/\m =\m/\m=0$.  Therefore, $I \subseteq I^{\rm F}_J \subseteq I^+_J \subseteq I^*_J = I$, making all terms equal.  But $u\in I^{\rm F}_R$, whence $u \in I^\cl_R$ for $\cl = {\rm F}, +, *$.  Thus, for each $\cl$, we have $u \in (I^\cl_R \cap J) \setminus I^\cl_J$.  For a concrete example, let $k$ be a field of characteristic $p \equiv 2$ (mod $3$), $R = k[\![x,y,z]\!] / (x^3 + y^3+ z^3)$, $I=(x,y)$, and $u=z^2$ \cite[Example 2.2]{Hu-tcparam}

    Then it follows from  \cite[Proposition 3.5]{nmeSc-tint} and smile duality that over any such ring, tight interior is not cohereditary. By the same token, since the smile dual of Frobenius closure of zero is $\tr_{R^\infty}$, it follows that module trace is also not cohereditary in general.
\end{example}

\begin{example}\label{ex:persprops}
We explore here the persistence property for various operations in Table~\ref{ta:exprops}. Let $\phi:R \to S$ be a map of rings.
\begin{enumerate}
    \item Radical is persistent.  If $x\in \sqrt{I}$, then $x^n \in I$ for some $n$, so $\phi(x)^n = \phi(x^n) \in IS$.
    \item Frobenius closure is persistent. If $z\in M$ with $z^q_M \in L^{[q]}_M$, then $(z \otimes 1)^q_{M \otimes_R S} \in (LS)^{[q]}_{M \otimes_R S}$.
    \item\label{it:Jbfpers} $J$-basically full closure is persistent, in the sense of Remark~\ref{rem:taprops}.  Let $z\in \Jcol J LM = (JL :_M J)$, with $L \subseteq M$ an $R$-module inclusion, $J$ an ideal of $R$, and $\phi: R \ra S$ a ring homomorphism.  Then for any $\alpha \in JS$, we have $\alpha =\sum_{i=1}^n j_i s_i$ with $j_i \in J$, $s_i \in S$, so that $j_i z \in JL$ for each $i$.  Write $j_i z = \sum_{\mu=1}^{t_i} h_{i\mu} \ell_{i\mu}$, $h_{i\mu} \in J$, $\ell_{i\mu} \in L$.  Then $\alpha \cdot (z \otimes 1) = \sum_i j_i z \otimes s_i = \sum_i \sum_\mu h_{i\mu} \ell_{i\mu} \otimes s_i = \sum_i \sum_\mu (h_{i\mu} \otimes 1) \cdot (\ell_{i \mu} \otimes s_i) \in (JS)(LS)$.  Since this holds for any $\alpha \in JS$, we have $z \otimes 1 \in ((JS)(LS) :_{M \otimes_RS} (JS)) = \Jcol {{(JS)}} {(LS)} {M \otimes_RS}$.
    \item $J$-basically empty interior is persistent in the sense of Remark~\ref{rem:taprops}.  Let $z\in L_{J\rm{be}}^M = J(L:_MJ)$, with $L \subseteq M$ an $R$-module inclusion, $J$ an ideal of $R$, and $\phi: R \ra S$ a ring homomorphism.  Then $z=\sum_{i=1}^n j_i \lambda_i$, with $j_i \in J$ and $\lambda_i \in (L :_M J)$, and by a similar argument as in (\ref{it:Jbfpers}), we have $\lambda_i \otimes 1 \in (LS :_{M \otimes_R S} JS)$ for each $i$.  Thus, $z\otimes 1 = \sum_i j_i\lambda_i \otimes 1 = \sum_i (j_i \otimes 1)(\lambda_i \otimes 1) \in (JS)(LS :_{M \otimes_R S} JS) = (LS)^{M \otimes_RS}_{(JS)\rm{be}}$.
    \item Tight interior is not persistent. Let $(R,\m)$ be a regular F-finite local ring and $J \subset \m$ a radical ideal of $R$ such that $S=R/J$ is not F-regular (e.g. $R=k[\![x,y,z]\!]$, $J = (x^3 + y^3 + z^3)$, $\chr k \neq 3$).  Then by \cite[Proposition 2.3]{nmeSc-tint}, $R_* = \tau_b(R)=R$ and $S_* = \tau_b(S) \subseteq \m_S$. Thus, $R_* S \nsubseteq S_*$, since $R_* S= RS = S$, whereas $S_* \subseteq \m_S$.
    \item The socle operation on pairs of modules over local rings $(R,\m)$ given by $p(L,M) = (L :_M \m)$, while functorial, is not persistent even across very reasonable ring maps.  Indeed, let $R=k[\![x]\!]/(x^2)$ and $S=R[\![y]\!] = k[\![x,y]\!]/(x^2)$.  Then $p(0,R) = \ann \m_R =xR$, but $p(0,S)= \ann \m_S = 0$, so in the inclusion map $i: R \hookrightarrow S$ we have $i(p(0, R)) \nsubseteq p(0,S)$.
\end{enumerate}
\end{example}

\subsection{Module closures and trace submodules}\label{sec:trace}
This section serves to introduce pair operation generalizations of module closures and trace. We demonstrate how applying our smile duality and limit to these operations allows us to prove theorems connecting trace and test ideals. 


\begin{defn}[{\cite[Definition 2.3]{RG-bCMsing}}]\label{def:mc} For an $R$-module $L$, the \emph{module closure} coming from $L$ is given by \[
N^{\cl_L}_M := \{ u\in M \mid \forall x\in L,\ x \otimes u \in \im (L \otimes N \rightarrow L \otimes M) \},
\]
whenever $M$ is an $R$-module and $N$ a submodule of $M$. If $L$ is an $R$-algebra, we refer to $\cl_L$ as an \textit{algebra closure}, in which case an equivalent characterization is given by 
\[N^{\cl_L}_M = \{ u\in M \mid 1 \otimes u \in \im (L \otimes N \rightarrow L \otimes M) \}.\]
\end{defn}

\begin{example} \label{ex:modclosure}
        Two simple examples of module closures are the trivial closure, which is given by setting $L=R$, and the improper closure, which sends every pair $(N,M)$ to $M$ and is given by setting $L$ to be the zero ideal.
\end{example}


 As we shall see, the following is a useful generalization:

\begin{defn}[cf. {\cite[Definition 8.15]{nmeRG-cidual}} for the associated submodule selector]\label{def:subsetmc}
If $S$ is a subset of an $R$-module $L$, then we set \[
N^{\cl_{S,L}}_M := \{u \in M \mid \forall s \in S,\ s \otimes u \in \im (L \otimes N \ra L \otimes M)\}.
\]
 We call this a \emph{subset module closure} (or \emph{subset algebra closure} if $L$ is an $R$-algebra). 
One important case of this is where $S = \{c\}$ is a singleton in $L$.  In this case, we write $\cl_{c,L}$.  When $L$ is an $R$-algebra this is called a \emph{subset algebra closure}.
\end{defn}

\begin{rem}\label{rem:smcnotidem}
Definition~\ref{def:subsetmc} really is a generalization of \ref{def:mc}, as it is easily seen that $\cl_L = \cl_{L,L}$, and if $L$ is an $R$-algebra we have $\cl_L = \cl_{1, L}$.  However,
 unlike $\cl_L$, $\cl_{S,L}$ is \emph{not} in general a closure operation, in that it may not be idempotent.  For instance, for any subset $S \subseteq R$ and ideal $I$ we have $I^{\cl_{S,R}}_R = (I :_R S)$, which is almost never an idempotent operation on ideals.
\end{rem}

\begin{prop}[See {\cite[Corollary 8.19]{nmeRG-cidual}}]\label{pr:smc}
    Let $L$ be an $R$-module and $S \subseteq L$ a subset. The subset module closure  $\cl_{S,L}$ is functorial, persistent, and residual.
\end{prop}

\begin{proof}
    First we show that $\cl_{S,L}$ is functorial. Let $f:M \to M'$ be an $R$-module homomorphism, $N$ a submodule of $M$, and $u \in N_M^{\cl_{S,L}}$. We will show that $f(u) \in f(N)_{M'}^{\cl_{S,L}}$.
    Since $u \in N_M^{\cl_{S,L}}$, for each $s \in S$, $s \otimes u \in \im(L \otimes N \to L \otimes M)$. Hence $s \otimes u$ can be written as $\sum_i \ell_i \otimes n_i$ for some $\ell_i \in L$ and $n_i \in N$. Applying $1 \otimes f:L \otimes M \to L \otimes M'$, we get 
    \[s \otimes f(u)=\sum_i \ell_i \otimes f(n_i) \in \im(L \otimes f(N) \to L \otimes M').\]
    Since this holds for each $s \in S$, $f(u) \in f(N)_{M'}^{\cl_{S,L}}$.

    Next we show that $\cl_{S,L}$ is residual. Let $P \subseteq N \subseteq M$ and $\pi:M \onto M/P$ be the natural surjection. Since $\cl_{S,L}$ is functorial, $\pi(N_M^{\cl_{S,L}}) \subseteq (N/P)_{M/P}^{\cl_{S,L}}$. Hence
    $N_M^{\cl_{S,L}} \subseteq \pi^{-1}((N/P)_{M/P}^{\cl_{S,L}}),$ 
    so it remains to prove the other inclusion. 
    Let $u \in \pi^{-1}((N/P)_{M/P}^{\cl_{S,L}})$, so that 
    \[s \otimes \pi(u) \in \im(L \otimes N/P \to L \otimes M/P) = \ker(L \otimes M/P \to L \otimes M/N),\]
    for every $s \in S$, with the second equality by right exactness of tensor.
    This implies that if $q:M/P \onto M/N$ is the quotient map, then $(1 \otimes q)(s \otimes \pi(u))=0$. That is, $s \otimes q(\pi(u))=0$. This implies that 
    \[s \otimes u \in \ker(L \otimes M \to L \otimes M/N)=\im(L \otimes N \to L \otimes M), \]
    again by right exactness of tensor, so that $u \in N_M^{\cl_{L,S}}$. Hence $\cl_{L,S}$ is residual.

    Finally for persistence (in the sense of Remark~\ref{rem:taprops}), let $\phi: R \ra T$ be a ring homomorphism, and $x\in N^{\cl_{S,L}}_M$.  Let $\mu \in ST$.  Then $\mu = \sum_{i=1}^k s_i \otimes t_i$, $s_i \in S$, $t_i \in T$.  In the natural isomorphism \[
    (L \otimes_R T) \otimes_T (M \otimes_R T) \stackrel{\simeq}{\longrightarrow} L \otimes_R M \otimes_R T,
    \]
    the element $\mu \otimes x \otimes 1$ maps to $\sum_{i=1}^k s_i \otimes x \otimes t_i$.  Since $x \in N^{\cl_{S,L}}_M$, we have $s_i \otimes x \in \im(L \otimes_R N \ra L \otimes_R M)$ for each $i$.  Therefore, $s_i\otimes x \otimes t_i \in \im(L \otimes_R N \otimes_R T \ra L \otimes_R M \otimes_R T)$, so that $\sum_i s_i \otimes x \otimes t_i$ is in that image.  Passing back through the isomorphism above, the image in question becomes $\im((L \otimes_R T) \otimes_T (N \otimes_R T) \ra (L \otimes_R T) \otimes_T (M \otimes_R T))$, which in turn is the result of tensoring the inclusion map $NT \hookrightarrow M \otimes_RT$ over $T$ with $L \otimes_R T$.  Thus, $x \otimes 1 \in (NT)^{\cl_{ST, L \otimes_R T}}_{M \otimes_RT}$.
\end{proof}

\begin{rem}
    For the module closure coming from an $R$-module $L$, the closure of 0 in a module $M$ is given by $0_M^{\cl_L}= \bigcap_{\ell \in L} \ker(M \to L \otimes M)$ where the map sends $x \mapsto \ell \otimes x$. This is a version of a torsion submodule, and indeed this operation gives a submodule selector since module closures are residual. 
\end{rem}

Many closure operations used in commutative algebra are either equal to module closures or are slight variants on them:
\begin{enumerate}
    \item Plus closure is an algebra closure where the algebra is the absolute integral closure $R^+$ of an integral domain $R$ \cite{HH-bCM,Sm-param}, \cite[Definition 2.7]{Die-TCfinlength}. 
    \item Frobenius closure is a direct limit of algebra closures for the algebras $F_*^e(R)$ for $e \geq 1$ 
    \cite[Section 10]{HHmain}, \cite[Section 6]{HHbase}. 
    \item Big \CM\ module (respectively algebra) closures are module (respectively algebra) closures with respect to a big \CM \ module (respectively algebra) $B$ \cite{Di-clCM}, \cite[Section 3.1]{RG-bCMsing}.
    \item\label{it:tc} Tight closure agrees with a closure coming from a big \CM\ algebra on \fg\ modules \cite[Section 14]{ho-solid}, and in general is defined by a meet of  subset algebra closures for the algebras $F_*^e(R)$  \cite[Section 8]{HHmain} (see Lemma \ref{lem:tclim}).   
    \item Almost big \CM\ algebra closures are algebra closures with respect to a subset $\{\pi^{1/p^n}\}_{n > 0}$ of an almost big Cohen-Macaulay algebra $T$ \cite{maschwedetestmultiplierideals,PeRG}, \cite[Section 11]{nmeRG-cidual}. 
    \item Full extended plus closure on an integral domain containing a prime $p$ in its Jacobson radical is a 
    meet of  subset algebra closures 
    modified by the ideals $(p^NR^+)_{N\geq 1}$  with respect to the set of multiples $\{c^\epsilon\}_{\epsilon>0}$ in the absolute integral closure $R^+$  \cite[Section 1]{Heit-ex+}, \cite[Defintion 2.3]{HeMa-extplus}.

\end{enumerate}

As a result, studying (subset) module and (subset) algebra closures gives us techniques that can be applied in all of these cases.
To be explicit about it (and since we will use it later), we detail (\ref{it:tc}) below:

\begin{defn}
Let $R$ be a reduced ring of prime  \charp.  There is a unique smallest \emph{perfect} (i.e. containing a unique $p$th root of every element) ring $R_\perf$ that contains $R$ (first shown in \cite{Gr-perfect}; notation as in \cite{BIM-regperf}).  We use $R^{1/p^e}$ to denote all $x\in R_\perf$ such that $x^{p^e} \in R$.  

Recall that $p$th roots are unique in prime characteristic reduced rings when they exist.  Thus for $c\in R$, we write $c^{1/p^e}$ for the unique element whose $p^e$th power is $c$. 
\end{defn}

\begin{defn}
[\cite{HHmain}; See {\cite[Definition 2.2(ii)]{Ta-fintestnum}} for this formulation] Let $R$ be a commutative Noetherian 
reduced ring of prime characteristic $p>0$. Let $R^\circ = \{c \in R \mid c$ is not in any minimal prime of $R\}$.  Let $L \subseteq M$ be $R$-modules. Then the \emph{tight closure of $L$ in $M$}, written $L^*_M$, consists of all those $z \in M$ such that there is some $c\in R^\circ$ such that \begin{equation}\label{eq:tcmod}
c^{1/p^e} \otimes z \in \im(R^{1/p^e} \otimes_R L \ra R^{1/p^e} \otimes_R M)
\end{equation} for all $e\gg 0$, where the map is induced by the inclusion map $L \hookrightarrow M$.

An element $c\in R^\circ$ is a \emph{big test element} if whenever $L \subseteq M$ are $R$-modules and $z \in L^*_M$, (\ref{eq:tcmod}) holds for that particular $c$ and all $e\geq 0$.
\end{defn}

\begin{thm}[\cite{HHbase}]\label{thm:testexist}  Let $R$ be a prime characteristic Noetherian ring that is essentially of finite type over an excellent local ring (e.g. any prime characteristic complete Noetherian reduced ring).  Then $R$ admits a big test element.
\end{thm}

Here is a first motivation for defining subset module closures:

\begin{lemma}\label{lem:tclim}
Let $R$ be a Noetherian prime characteristic reduced ring that admits a big test element $c$.  Let $*$ be the tight closure operation.  Then $* = \displaystyle \bigwedge_{e\in \N} \cl_{c^{1/p^e}, R^{1/p^e}}$.
\end{lemma}

\begin{proof}
    Unroll the definitions.
\end{proof}


Next we define trace modules and discuss their appearances in the literature.

\begin{defn}[{\cite[Definition 8.1]{nmeRG-cidual}}]
    Let $L$ be an $R$-module and $S$ a subset of $L$. The \textit{$(S,L)$-trace} of an $R$-module $N$ is the submodule selector
    \[\tr_{S,L}(N) = \sum_{f \in \Hom_R(L,N)} Rf(S).
    \]
    That is, $\tr_{S,L}(N)$ is the submodule of $N$ generated by the set $\{f(s) \mid f \in \Hom_R(L,N), s \in S\}$.  If $S$ is a sub\emph{module} of $L$, it is equivalent to write \[
    \tr_{S,L}(N) = \im (S \otimes_R \Hom_R(L,N) \ra N),
    \]
    where the map is given on simple tensors by $s \otimes f \mapsto f(s)$.

When $S=L$, we write $\tr_L$ for $\tr_{S,L}$ and we speak of the \emph{$L$-trace} of an $R$-module.  
When $N=R$, this is known as the \emph{trace ideal} of the $R$-module $L$ \cite{Lam99,Lin17}.





\end{defn}

\begin{rem}\label{rem:smtnotidem}
    As with module closures, 
    $\tr_L$ is an interior operation, whereas $\tr_{S,L}$ typically fails to be idempotent and is not an interior operation.
    In particular, for $I$ an ideal, $\tr_{I,R}(M) = IM$, so $\tr_{I,R}$ is only idempotent when $I=I^2$.
\end{rem}

\begin{prop}\label{pr:trNak}
Let $(R,\m)\ra (S,\n)$ be a local homomorphism of Noetherian local rings, and $L$ a finitely generated $S$-module.  Then $\tr_L$ is an absolute Nakayama interior on the category of Artinian $R$-modules.  In particular, this holds when $L$ is a finitely generated $R$-module, or when $L=S$ is a local $R$-algebra.
\end{prop}

\begin{proof}
Let $A \subseteq C$ be Artinian $R$-modules. Since $\tr_L$ is absolute, we omit reference to any ambient module $B$ containing both $A$ and $C$.  Suppose $\tr_L(A :_C\m) \subseteq A$.  First, note that $\Hom_R(L,C)$ is Artinian as an $S$-module.  Indeed, $L$ being finitely generated over $S$ means there is a surjection $S^n \onto L$, which translates into an injection $\Hom_R(L,C) \hookrightarrow \Hom_R(S^n, C) \cong \Hom_R(S,C)^n$, so that since direct sums and submodules of Artinian modules are Artinian, the claim holds.

Now let $\phi \in \Hom_R(L,C)$; assume $\im \phi \nsubseteq A$. The natural $S$-action on $\Hom_R(L,C)$ is given by $(s\phi)(\ell) := \phi(s\ell)$.  By the left-exactness of $\Hom$, we may consider $\Hom_R(L,A)$ to be an $S$-submodule of $\Hom_R(L,C)$, with $\phi \notin \Hom_R(L,A)$.  Then $H := \frac{S\phi + \Hom_R(L,A)}{\Hom_R(L,A)}$ is a nonzero $S$-subquotient of $\Hom_R(L,C)$.  Thus, it is Artinian as an $S$-module.  Since it is also cyclic as an $S$-module (generated by the image of $\phi$), there is some $k\geq 1$ such that $\n^k H=0$ but $\n^{k-1}H \neq 0$.  Let $t\in \n^{k-1}$ such that $tH \neq 0$.  Let $\psi = t\phi$.  Then $\im \psi \nsubseteq A$.

For any $m\in \m$ and $\ell \in L$, we have $m\psi(\ell)=m\phi(t\ell) = \phi(mt\ell) = (mt\phi)(\ell) \in A$ since $mt \in \m\n^{k-1} \subseteq \n^k \subseteq \ann H$.  Thus, $\im \psi \subseteq (A:_C\m)$, so there is a commutative diagram of $R$-linear maps \[
\xymatrix{
& (A :_C\m) \ar[dr]^i\\
L \ar[rr]_\psi \ar[ur]^{\widetilde \psi}& & C,
}
\]
where $i: (A :_C\m) \ra C$ is the inclusion map. Thus, $\im \psi = \im \tilde \psi \subseteq \tr_L(A :_C \m) \subseteq A$ by assumption.  But we saw above that $\im \psi \nsubseteq A$, providing a contradiction.  Therefore, the image of any $\phi \in \Hom_R(L,C)$ must be contained in $A$.  That is, $\tr_L(C) \subseteq A$, completing the proof that $\tr_L$ is a Nakayama interior.
\end{proof}

%

Trace ideals come up in a number of places in the literature.

\begin{example}
    Let $R$ be a \CM\ local ring with canonical module $\omega$. Since $\omega$ is free if and only if $R$ is Gorenstein, $\tr_\omega(R)=R$ exactly when the ring is Gorenstein.
\end{example}

This leads to the following class of rings:

\begin{defn}[Herzog-Hibi-Stamate 2019]
    Let $(R,\m)$ be a \CM\ local ring with canonical module $\omega$. Then $R$ is \emph{nearly Gorenstein} if $\tr_\omega(R) \supseteq \m$.
\end{defn}

A number of papers describe classes of rings with the nearly Gorenstein property, including \cite{hibistamategorensteinonpuncturedspectrum,hallkolblmatsushitamiyashitanearlygorensteinpolytopes,miyashitalevelness,miyazakipageehrhartrings,miyashitageneralizedgorensteinproperties,miyashitanearlygorensteinprojectivecurves,ficarraherzogstamatetrivedi,kumashiromatsuokanakashimamaximalminors,miyashitavarbaro,moscariellostrazzanti,hibistamatefinitegraphs,caminatastrazzanticyclicquotientsingularities}.  Other papers study the canonical trace, including \cite{herzogshinyacolengthofcanonicaltrace,ficarracanonicaltrace,herzogmohammadipage,daokobayashitakahashi,herzoghibistamatenumericalsemigroup}. Additional work compares the nearly Gorenstein property to the almost Gorenstein property, for example \cite{kumashirereductionnumber}. Herzog, Hibi, and Stamate also define and study far-flung Gorenstein singularities, when the canonical trace is as small as possible \cite{herzogshinyastamatefarflunggorenstein}.


Other appearances of trace ideals in the literature include:
\begin{enumerate}
    \item Faber studies trace ideals of \fg\ \CM\ modules over rings of finite \CM\ type \cite{FaberTraceIdeals}.
    \item Lindo and Pande prove that every ideal is a trace ideal if and only if the ring is Artinian Gorenstein \cite{LindoPande}. This work is extended by Goto, Isobe, and Kumashiro \cite{gotoisobekumashiro}, who classify when every submodule of a given module is trace.
    \item Lindo applies trace ideals to study rigid modules in \cite{Lin17}, and follows up to connect trace modules to endomorphism invariant modules and reflexivity in \cite{lindothompson}. Lindo further applies trace modules to a case of the Auslander-Reiten Conjecture in \cite{lindononrigidity}.
    \item Dao and Lindo use stable trace ideals in \cite{daolindo} to study reflexive modules and Arf rings.
    \item Kumashiro determines when a ring has a finite number of trace ideals, connecting this to the reflexive modules and their endomorphism algebra \cite{kumashirofinitetraceideals}. Kobayashi and Kumashiro also examine cases where a ring has a finite number of trace ideals in \cite{kobayashikumashiro}.
    \item Maitra works with a variant called a partial trace ideal and its connection to modules of differentials beginning in \cite{MaitraPartialtraceideals}.
    \item Benali, Pothagoni, and the second-named author compute trace ideals of maximal \CM\ modules over ADE Singularities in \cite{benalipothagoniRG}, given that these are the test ideals of the corresponding module closures.
    \item McDonald proves that the multiplier ideal is a sum of trace ideals in \cite{Mcd-multklt}.
\end{enumerate}   

\begin{notation}
     We defined $\tr_{S,L}$ above as a submodule selector. However, we can define an absolute pair operation as in Section \ref{sec:subselect} by
    \[\tr_{S,L}(N,M):=\tr_{S,L}(N).\]

    Since the ambient module does not matter in absolute pair operations, from this point on we omit mention of the ambient module $M$ when using $\tr_L$ and $\tr_{S,L}$.
\end{notation}

The following restates \cite[Lemma 8.4]{nmeRG-cidual} in the context of pair operations. This adaptation relies on the results of Section \ref{sec:subselect} that indicate how properties of a submodule selector $\alpha$ translate to properties of $g(\alpha)$.

\begin{lemma}[See {\cite[Lemma 8.4]{nmeRG-cidual}}]
\label{lem:algtrace}
Let $L$ be an $R$-module and $S$ a subset of $L$.  Then $\tr_{S,L}$ is a functorial, absolute, intensive pair operation that is order-preserving on submodules, which is idempotent if $S=L$. Consequently, when $S=L$ it is a functorial, absolute interior operation.
\end{lemma}

In order to work with closures coming from families of $R$-modules or algebras, such as integral closure and closures coming from a family of big \CM\ modules or algebras, we combine the idea of trace with direct and inverse limits.

The following condition helps us relate different traces of the same module.

\begin{defn}
Let $L$ and $M$ be $R$-modules. We say that $L$ \textit{generates} $M$ if some direct sum of copies of $L$ surjects onto $M$. In particular, $L$ generates $M$ if there is a surjection $L \twoheadrightarrow M$.
\end{defn}

The following lemma is well-known, but appears in particular in \cite[Proposition 2.8]{Lin17}. That paper states the result in a narrower context, but the proof works in our more general setting.

\begin{lemma}
If $L \twoheadrightarrow L'$ is a surjection of $R$-modules, or more generally if $L$ generates $L'$, then $\tr_{L'} \le \tr_L$. 
\end{lemma}

\begin{example}
    Given $R$-algebras $A$ and $B$ and an $R$-algebra map $A \to B$, $A$ generates $B$, so $\tr_B \le \tr_A$.
    This is applied in \cite{PeRG} to study test ideals of directed families of big \CM\ $R$-algebras.
\end{example}








Module closures and the trace operations are dual to each other, even with respect to subsets:
\begin{thm}\label{thm:tracedual}
Let $R$ be a complete local Noetherian commutative ring, $L$ an $R$-module, and $S \subseteq L$ a subset.  Then for $ (N,M) \in \cP$ as in Remark \ref{ourpairs}, 
$\tr_{S,L}^\dual(N,M) = N_M^{\cl_{S,L}}$. 
\end{thm}

That is, \emph{the $(S,L)$-trace is the interior operation dual to the subset module closure given by $(S,L)$}. This demonstrates how our framework can be used to achieve results comparable to those in \cite{PeRG}.

\begin{proof}
    First, let $x \in N_M^{\cl_{S,L}}$. Then for every $s \in S$, $s \otimes x \in \im(L \otimes_R N \to L \otimes_R M)$. By Proposition \ref{pr:kernelview}, it suffices to show that $g(x)=0$ for every $g \in \tr_{S,L}((M/N)^\vee,M^\vee)=\tr_{S,L}((M/N)^\vee)$.
Recall from Definition~\ref{def:nonresidualdual} that we identify $(M/N)^\vee$ with the submodule $U := \{h \mid h(N) = 0\}$ of $M^\vee$.

Accordingly, let $g \in \tr_{S,L}(U)$. Then there exist maps $\phi_i: L \ra U$ and elements $s_i \in S$ such that $g= \sum_i \phi_i(s_i)$. By viewing the $\phi_i$ as maps $L \to M^\vee$, we can use Hom-tensor adjointness to define maps $\psi_i: L \otimes_R M \ra E$ such that $\psi_i(\ell \otimes y) = \phi_i(\ell)(y)$ for each $\ell \in L$, $y \in M$. 
For each $i$, there exist $\lambda_{ij} \in L$, $n_{ij} \in N$ such that $s_i \otimes x = \sum_j \lambda_{ij} \otimes n_{ij}$.  Thus, \[
g(x) = \left(\sum_i \phi_i(s_i)\right)(x)=\sum_i \psi_i(s_i \otimes x) = \sum_i \psi_i\left(\sum_j \lambda_{ij} \otimes n_{ij}\right) = \sum_{i,j} \psi_i(\lambda_{ij} \otimes n_{ij}) = \sum_{i,j} \phi_i(\lambda_{ij})(n_{ij})=0,
\]
since for each pair $i,j$ we have $\phi_i(\lambda_{ij}) \in U$, and is hence a function that vanishes on $N$.

    For the other inclusion, let $x \in \tr_{ S,L}^\dual(N,M)$. Then for every $g:M/N \to E$ that can be written as $\sum_i \phi_i(s_i)$ (as in the other direction), $g(\bar{x})=0$. By Hom-tensor adjointness, for every map $\psi:L \otimes M/N \to E$ and every $s \in S$, $\psi(s \otimes \bar x)=0$. Since $\Hom_R(-,E)$ doesn't kill a nonzero module, for every $s \in S$, $s \otimes x \in \ker(L \otimes M \to L \otimes M/N)=\im(L \otimes N \to L \otimes M)$.
\end{proof}

\begin{rem}
    Note that the $R$-module $L$ has no restriction on it -- it does not have to be either \fg\ or artinian.
\end{rem}

Combining this with Theorem \ref{thm:finintideal}, we get the following slight generalization of a result in the literature:

\begin{thm}[See {\cite{PeRG,ERGV-chdual}}]
    Let $(R,\m,k)$ be a complete Noetherian local ring, $E = E_R(k)$, $L$ an $R$-module, and $S \subseteq L$. Then $\tr_{S,L}(R)=\ann_R(0_E^{\cl_{S,L}})$, i.e. $\tr_{S,L}(R)$ is the test ideal (in the version defined by $E$) for the subset module closure $\cl_{S,L}$.
\end{thm}

Note that this follows from applying Theorem \ref{thm:finintideal}, and this can be done for finitistic test ideals as well.

\begin{example}
    This result is applied in \cite{benalipothagoniRG} to compute the test ideals of \CM\ module closures, as trace ideals for \fg\ modules can be computed in Macaulay2.   In \cite{EMRS}, it is applied to prove that the alterations closure has test ideal equal to its multiplier ideal.
\end{example}

Applying these results to tight closure, we get the following result, which was first shown in \cite{DEST2}, but with different methods than we are using here: 

\begin{thm}
Let $R$ be a complete Noetherian local reduced ring of prime \charp, and $c$ a big test element.  Then \[
\tau_*(R) = \left(\bigvee_{e\in \N} \tr_{c^{1/p^e}, R^{1/p^e}}\right)(R) =  \sum_{e \ge 0} \sum_{\psi\in \Hom_R(R^{1/p^e},R)} \psi((cR)^{1/p^e}).
\]
\end{thm}

\begin{proof}
    The second equality is a result of unrolling definitions.  For the first equality, from Theorem~\ref{thm:finintideal} we have $\tau_*(R) = \ann 0^*_E = R^R_{*^\dual}$. But by Lemma~\ref{lem:tclim}, Proposition~\ref{pr:limduals}, and Theorem~\ref{thm:tracedual} we have \[
    *^\dual = \left(\bigwedge_{e\in \N} \cl_{c^{1/p^e}, R^{1/p^e}}\right)^\dual = \bigvee_{e\in \N} (\cl_{c^{1/p^e}, R^{1/p^e}})^\dual = \bigvee_{e\in \N} \tr_{c^{1/p^e}, R^{1/p^e}}. \qedhere
    \]
\end{proof}



Compare the above to the following more familiar result, with identical conclusion but quite different hypotheses on the ring.

\begin{thm}[\cite{HaTa-gentest}, with notation from \cite{ScTu-survey} and Definition~\ref{def:testideals}]
    Let $R$ be an integral domain essentially of finite type over a perfect field of \charp. Fix a  sufficiently nice test element $c \in R$. Then
    \[\tau_*(R)=\sum_{e \ge 0} \sum_{\psi\in \Hom_R(R^{1/p^e},R)} \psi((cR)^{1/p^e}).\]
\end{thm}






\subsection{Core-hull duality and basically full closures}
\label{sec:corehull}


In this subsection we discuss the $\cl$-core, which is a generalization of the integral closure core, and its dual hull.  
We mention some of the many investigations into core in the literature. In particular, existing work has touched on the pair operation properties discussed in this paper. In addition, we provide an introduction to the $J$-basically full closure and $J$-basically empty interior, which relate to the only known formulas for the hull of a submodule of the injective hull of the residue field.

The (integral closure) core was originally introduced by Rees and Sally  \cite{ReSa-core} in their proof of the Brian\c{c}on-Skoda Theorem which states that for a $d$-dimensional regular local ring, the integral closure of the $d$th power of any ideal $I$ is contained in $I$.  Rees and Sally actually proved the stronger statement that the integral closure of the $d$th power of $I$ is contained in all the reductions of $I$ and hence in the core of $I$:  \[{\rm core}(I):= \bigcap\{ J \subseteq I \mid J \text{ an (integral) reduction of } I\}.\]  
Below we detail some applications of core in the literature.
\begin{enumerate}
    \item  Adjoint or multiplier ideals have been one tool commonly used for classifying singularities.  Early on in the study of core of ideals, Huneke and Swanson \cite[Proposition 3.14]{HuSw-core} noted the relationship between the core of ideals and adjoint (or multiplier) ideals.  They showed that in a 2-dimensional regular local ring,  ${\rm core}(I)={\rm adj}(I^2)$ for ideals $I$.  Hyry and Smith \cite{HySmi-Kawcore} generalized this formula to $\m$-primary ideals $I$ in a $d$-dimensional Gorenstein local ring of essentially finite type over a field of characteristic 0 with $R[It]$ having rational singularities, i.e. ${\rm core}(I)={\rm adj}(I^d)$. 
    \item  The core of powers of the graded maximal ideal $\m$ of a standard graded ring $R$ strikingly has ties to geometric properties of $\text{Proj}(R)$. Notably, Hyry and Smith \cite{HySmi-Kawcore}, \cite{HySmi-cvgc} investigate a variant of the core, the graded core:  \[{\rm gradedcore}(I)=\bigcap\{ J \subseteq I \mid J \text{ a homogeneous (integral) reduction of } I\}.\]  Assuming $a$ is the $a$-invariant of $R$, they show that if $${\rm gradedcore}(\m^n)={\rm core}(\m^n)=\m^{nd+a+1},$$ then Kawamata's conjecture holds. Kawamata's conjecture asserts that every nef line bundle adjoint to an ample line bundle over a smooth projective variety admits a nonzero section.  It is known that ${\rm core}(\m^n)=\m^{dn+a+1}$ if and only if $[\omega]_{-a}$, the submodule of the canonical module generated in fixed degree $-a$ is faithful by work of \cite{HySmi-Kawcore}, \cite{HySmi-cvgc} and \cite{FPUgradanncore}.  Showing that Kawamata's conjecture holds amounts to showing that the core and gradedcore of powers of the maximal ideal are equal.  
    \item  Another geometric property related to the core of powers of the homogeneous maximal ideal of a standard graded rings is the Cayley-Bacharach Property: If $X$ is a finite set of reduced points in projective space, then the Hilbert function of $X$ with respect to any point $P \in X$ does not depend on $P$.  Fouli, Polini and Ulrich show in \cite{FPUgradanncore}, that if $R$ is the homogenous coordinate ring of such a finite set $X$, with graded maximal ideal $\m$, ${\rm core}(\m)=\m^{a+2}$ if and only if $X$ has the Cayley-Bacharach Property.
\end{enumerate}

 We begin by generalizing reductions and core to general closure operations.

\begin{defn}[compare {\cite[Section 2]{ERGV-chdual} or \cite[Definition 3.2]{FoVa-core}}]
Let $R$ be a commutative ring and $\cl$ a closure operation defined on a class $\cP$ of pairs $(N,M)$ as in Remark \ref{ourpairs}.
We say that $L \subseteq N$ is a \emph{$\cl$-reduction of $N$ in $M$} if 
$L^{\cl}_M=N^{\cl}_M$.
Note that 
$L \subseteq N \subseteq L_M^\cl$ if and only if $L$ is a $\cl$-reduction of $N$ in $M$.

The \emph{$\cl$-core of $N$ with respect to $M$} is the intersection of all $\cl$-reductions of $N$ in $M$, or
\[\cl\core_M(N):= 
\bigcap \{L \mid L \subseteq N \subseteq L_M^{\cl} 
\}.\]
When taking the $\cl$-core of an ideal $I$ in $R$, we will denote $\cl\core_R(I)$ by $\cl\core(I)$.
\end{defn}

\begin{example} Here are some straightforward examples demonstrating $\cl$-core of a submodule:  
  \begin{enumerate}
      \item If $N$ is the only $\cl$-reduction of $N$ in $M$, then $\cl\core_M(N)=N$.
      \item If $\cl$ is the identity closure, then  $\cl\core_M(N)=N$ for any submodule $N \subseteq M$.
      \item If $\cl$ is the improper closure operation as in \ref{ex:modclosure}, then $\cl\core_M(N)=0$ for all submodules $N \subseteq M$.
  \end{enumerate}  
\end{example}

 Many authors have considered basic properties of the $\cl$-core and found formulas for computing it in particular situations:
     \begin{enumerate}
        \item Computing the core can be hard.  Hence, one direction is to exhibit that the core is a finite intersection of reductions. See for example \cite[Theorem 3.1, Theorem 3.2 and Theorem 4.5]{CPU-strcore}, \cite{FMedgecore} and \cite{FMPU-coremon}.
        \item For some ideals in specialized rings, formulas are known to compute the core. Some formulas listed below depend on the characteristic. See for example \cite{Fou-corchar}, \cite{FPU-corechar}, \cite{FMedgecore}.
        \begin{itemize}
        \item Given a minimal reduction $J$ of $I$, there often exists $n$ depending on various quantities associated to $I$ and $J$ (such as height, grade and reduction number) so that one of the two formulas holds: $\text{core}(I)=J^{n+1}:I^n$.  or $\text{core}(I)=J(J^n:I^n)$.   See for example \cite{HuSw-core}, \cite{HT-core}, \cite{CPU-coreres}, and \cite{PU-core}.
            \item It is sometimes the case that the core of an ideal $I$ is the adjoint of a power of $I$. See for example \cite{HuSw-core}, \cite{HySmi-Kawcore}, \cite{HySmi-cvgc}, and \cite{DCruzetal-coreadjoint}.
            \item Let $\m$ be the (homogeneous) maximal ideal of a (graded) local ring $R$. The core of $\m^n$ is often a power of $\m$ depending on $n$, the dimension of $R$ and the $a$-invariant of $R$.  See \cite{HuSw-core},  \cite{HySmi-Kawcore}, \cite{HySmi-cvgc}, \cite{FPU-corechar}. 
            \item  The core of a monomial ideal $I$ in a polynomial ring is a monomial ideal.  Effective methods to compute the core of monomials can be found in \cite{PUV0mon}, \cite{Sm-corestrongstable}, \cite{FMedgecore}, \cite{Ko-zdmoncore}, and \cite{FMPU-coremon}.
        \end{itemize}
        \item Several authors have considered whether ${\rm core}_R(I)S={\rm core}_S(IS)$ for extensions $S \supseteq R$.  See for example \cite{CPU-strcore}, \cite{PUV0mon}, \cite{KMOPrufcore1},  and \cite{KMOPrufercore}.  As mentioned in Section \ref{sec:bigtable} when discussing persistence, it is known that the equality does not always hold; an example was given in \cite{HT-core}.
        \item When $J \subseteq I$ are integrally closed ideals, in some cases it is known that ${\rm core}(I)$ is also integrally closed and ${\rm core}(J) \subseteq {\rm core}(I)$.  See \cite{HuSw-core}, \cite{HySmi-Kawcore}, \cite{PU-core} and \cite{Okuma-coreratsing}.  There are however examples of integrally closed ideals whose core is not integrally closed.  See \cite[Example 4.13]{PUV0mon}. 
        \item In the non-Noetherian setting, work has been done to classify ideals which have small integral core, i.e. ${\rm core}(I)=I^2I^{-1}$.  See for example \cite{KMcoredom}, \cite{KMcore1dim}, and \cite{KMOPrufcore1}.
        \item Most of the work on integral care has been done for ideals, but there has been some progress determining the core of modules in \cite{Mo-coremod}, \cite{CPU-corpd1}, and \cite{CFH-coremod}.
        \item For a Nakayama closure $\cl$, the $\cl\core$ was discussed in \cite{FoVa-core}, \cite{ERGV-chdual}, \cite{ERGV-nonres} and \cite{ERGV-extend}. Also a formula for the tight closure core of certain ideals in rings of characteristic $p>0$ was given in \cite{FVV*core} to be $J(I:J)$ where $J$ is a minimal $*$-reduction of $I$.  
    \end{enumerate}


The dual notion of $\cl$-core is the $\cl^\dual$-hull, where $\cl^\dual$ is the interior operation which is dual to $\cl$.

\begin{defn}[{\cite[Section 6]{ERGV-chdual}}]\label{def:absintexp}
Let $R$ be a commutative ring and $\cP$ be a class of pairs $(A,B)$ as in Remark \ref{ourpairs}
.  
Let $\intr$ be an interior operation on $\cP$.  
We say $C$ with $A\subseteq C \subseteq B$ is an \emph{$\intr$-expansion of $A$ in $B$} if $A^B_{\intr}=C^B_{\intr}$.

The \emph{$\intr$-hull of a submodule $A$ with respect to $B$} is the sum of all $\intr$-expansions of $A$ in $B$, or
 \[
\intr\hull^B(A):=\sum_{\intr(C) \subseteq A \subseteq C \subseteq B} C.
\]
\end{defn}

\begin{example} Here are some straightforward examples demonstrating $\intr$-hull of a submodule:  
  \begin{enumerate}
      \item If $N$ is the only $\intr$-expansion of $N$ in $M$, then $\intr\hull^M(N)=N$.
      \item If $\int$ is the identity interior, then  $\intr\hull^M(N)=N$ for any submodule $N \subseteq M$.
      \item If $\intr$ is the trivial interior operation (i.e. the one sending all submodules to 0), then $\intr\hull^M(N)=M$ for all submodules $N \subseteq M$.
      \item If $(R, \m)$ is a local ring and $\intr$ is the interior operation on the set of ideals of $R$ defined by $$\intr^R(I)=\begin{cases}R \text{ if } I=R \\
      0 \text{ if } I \neq R.  \\
      \end{cases}$$
      then $\intr\hull^R(I)= \m$ for any proper ideal $I$.
  \end{enumerate}  
\end{example}

The main references for examples of $\intr\hull$ are \cite{ERGV-chdual}, \cite{ERGV-nonres}, and \cite{ERGV-extend}.  In fact, as long as $(R,\m)$ is a complete local ring, and $I$ is an ideal of $R$, we can use the following theorem to determine the integral hull or $*$-hull of $(R/I)^\vee$ in the injective hull of the residue field.  We detail some known formulas for hull in Theorem \ref{thm:hullform} and $*$-hull in Theorem \ref{thm:*hullform} below, both of which are consequences of the following duality theorems.

\begin{thm}[{cf. \cite[Theorem 6.3]{ERGV-chdual} \cite[Theorem 6.2]{ERGV-nonres}}]\label{thm:expreddual}
 Let $(R, \m)$ be a Noetherian complete local ring.  Let $\intr$ be a relative interior operation on a class of pairs $\cP$ as in Remark \ref{ourpairs}, and let $\cl := \intr^\dual$ be its dual closure operation.  There exists an order reversing one-to-one correspondence between the poset of $\intr$-expansions of $A$ in $B$ and the poset of $\cl$-reductions of $(B/A)^{\vee}$ in $B^{\vee}$. Under this correspondence, an $\intr$-expansion $C$ of $A$ in $B$ maps to $(B/C)^\vee$, a $\cl$-reduction of $(B/A)^\vee$ in $B^\vee$.
\end{thm}

\begin{thm}[{cf. \cite[Theorem 6.17]{ERGV-chdual} \cite[Theorem 6.6]{ERGV-nonres}}]
\label{thm:hullexists}
Let $(R,\m)$ be a Noetherian complete  local ring.  Let $A \subseteq B$ be Artinian $R$-modules, and let $\intr$ be a  Nakayama interior defined on Artinian $R$-modules. Then the $\intr$-hull of $A$ in $B$ is dual to the $\cl$-core of $(B/A)^\vee$ in $B^\vee$, where $\cl$ is the closure operation dual to $\intr$.
\end{thm}

Under various hypotheses on the ring and the ideal, various formulas have been given for the (integral) core of an ideal.  Without going into the hypotheses, the following gives a formula for the core of an ideal $I$ in terms of any minimal reduction $J$ of $I$ where $r$ is the reduction number of $I$ with respect to $J$. 
\[
{\rm core}(I)=I(J^r:_RI^r)=J(J^r:_RI^r)=(J^{r+1}:_RI^r).
\]
 The hypotheses can be found in \cite[Proposition 5.3]{CPU-coreres}, \cite[Theorem 3.7]{HT-core}, \cite[Theorem 4.5]{PU-core}, \cite[Theorem 3.3]{FPU-corechar} where the power can be relaxed to any $n \geq r$ in some cases.


 \begin{thm} \cite[Theorem 7.9]{ERGV-nonres}
 \label{thm:hullform}
If $(R,\m)$ is a complete local ring 
and $I$ is an ideal satisfying ${\rm core}(I)=(J^{n+1}:I^n)$ for some
reduction $J$ and natural number $n$,
then
\[
{\rm hull}^E(0:_E I)=I^n(0:_E J^{n+1}).
\]
\end{thm}

 These formulas for (integral) core and hull are related to the following closure and interior operations:
\begin{defn} \label{def:bee} \cite[Definition 4.1 and 4.9]{ERGV-nonres}
    Let $R$ be a commutative ring and let $J$ be an ideal of $R$.  Then for any submodule inclusion $L \subseteq M$:
    \begin{enumerate}
        \item The \textit{$J$-basically full closure} of $L$ in $M$ is given by $
\Jcol JLM := (JL :_M J).$
\item The \emph{$J$-\bemp \  interior} of $L$ in $M$ is $\Jintrel JML := J (L :_M J).$
    \end{enumerate}
\end{defn}
These were originally inspired by the work of Heinzer, Ratliff, and Rush on basically full closure \cite{HRR-bf} as well as \cite{VaVr-tctest} and \cite{Rush-mfull} and the following works \cite[Definition 3.7]{CIST-AuslanderReiten} and \cite[Definition 3.1]{Dao-colon}. When $J=\m$, 
some authors refer to an $\m$-basically full closed ideal as \emph{weakly $\m$-full}. 
 However, the $J$-basically full closures also come up in \cite[Proposition 1.58]{Vasc-icbook} as a test for membership in the integral closure of an ideal $I$ of an integral domain. (See also \cite[Theorem 2.4]{CHVIntegralclosure}, which says that if $I$ is generically a complete intersection and $J$ is the Jacobian ideal of $I$ in a polynomial ring, then $I$ is integrally closed $\iff I_R^{J{\rm bf}}=I$.) 
 Further, the $J$-basically full condition comes up in \cite[Section 2]{CorsoPoliniS2} for $J=I^t$ in the $s$-generated ideal $\mathfrak{a}$ which defines an $s$-residual intersection $K=\mathfrak{a}:I$.  In particular, $\mathfrak{a}^{I^t{\rm bf}}_R=\mathfrak{a}$ for $1 \leq t \leq s-g$ where $g$ is the grade of $I$.
 
 The $J$-basically empty interiors also appear in \cite[Lemma 3.7]{DMS-reflexiveUlrich} where the authors define the  trace ideal of a regular ideal in terms of colons ($\text{tr}_R(I)=((x)^R_{I{\rm be}}:x)$ for a regular ideal $I$ an nonzerodivisor $x \in I$)  and in \cite[Lemma 3.3]{DeK-Burch}, \cite[Definition 2.1, Proposition 2.3]{DKT-Burch} 
 in defining a submodule of a module to be Burch ($N \subseteq M$ is Burch if $N^M_{\m{\rm be}} \neq  \m N$). 



Fouli, Vassilev and Vraciu  \cite{FVV*core} devised a formula for the $*\core$ (where $* =$ tight closure) of some ideals in normal local rings of characteristic $p>0$.  They discovered three sufficient
conditions \cite[Theorems 3.7, 3.10 and 3.12]{FVV*core} such that the $*\core (I) =I(J :_R I)$, giving us the
following Proposition:

\begin{prop}\label{cor:*corebe} \cite[Proposition 7.13]{ERGV-nonres}
Let $(R, \m)$ be a normal local ring of characteristic $p>0$ with perfect residue field.  Let $\tau = \tau_*^{fg}(R)$ be the finitistic tight closure test ideal. Suppose one of the following holds:
\begin{enumerate}
    \item $R$ is Cohen-Macaulay and excellent with 
    $\dim R \geq 2$.  Let $x_1, x_2,\ldots,x_d$ be part of a system of parameters and $J=(x_1^t,x_2^t,x_3, \ldots,x_d)$ where $x_1, x_2 \in \tau$ and $t \geq 3$.  Let $J \subseteq I \subseteq J^*$.
    \item The test ideal $\tau=\m$ and $J$ is any minimal $*$-reduction of $I$.
    \item The test ideal $\tau$ is $\m$-primary and $J'$ is a minimal $*$-reduction of $I'$ and $J=(J')^{[q]}$ and $I=(I')^{[q]}$ for large $q=p^e$.
\end{enumerate}
Then \[*\core (I)=I(J:_R I)=\Jintrel{I}{R}{J}.\]
\end{prop}

 \begin{thm}\label{thm:*hullform} \cite[Theorem 7.14]{ERGV-nonres}
Let $(R, \m)$ be a complete normal local ring of characteristic $p>0$ 
and $I$ is an ideal satisfying $*\core_R(I)=I(J:_R I)$ for some $*$-reduction $J$.
Then \[*\hull^E(0:_E I)=\Jcol I {(0:_E J)}{E}.\]  
\end{thm}

\section*{Acknowledgments}

We thank Mel Hochster and Craig Huneke for their guidance as thesis advisors and their inspiring work on tight closure which eventually led us to our work on pair operations. We thank Alessandra Costantini and Karl Schwede as members of the community who asked interesting questions during the course of our papers. We are also grateful for discussions throughout the years with Jesse Elliott, who had a unique perspective on closure operations and whose recent passing is a distinct loss to the community. We thank the anonymous referee who made many suggestions that improved the paper, including the suggestion to include something like Table~\ref{ta:exprops}. Finally, we thank Abdullah Alshayie and Sarah Poiani who, as our students, are pushing pair operations in new directions, and for their careful reading of our papers.

\appendix \section{Some pair operations}

The purpose of this appendix is to provide definitions for any pair operations in Table~\ref{ta:exprops} that are not defined elsewhere in the paper.

\begin{plaindef}
We start with three pair operations that are defined everywhere.  \begin{itemize}
    \item The operation $p(L,M) := M$ for all $L \subseteq M$ is the \emph{indiscrete}, or \emph{improper closure} 
    operation.  It is the largest pair operation and the largest closure operation.
    \item The operation $p(L,M) := L$ for all $L \subseteq M$ is the \emph{identity} pair operation, or \emph{trivial closure} operation. 
    It is simultaneously the smallest closure operation and the largest interior operation.
    \item The operation $p(L,M) := 0$ for all $L \subseteq M$ is the \emph{zero} operation. It is the smallest pair operation and the smallest interior operation.
\end{itemize} 
\end{plaindef}

\begin{plaindef}
Let $R$ be a ring and $I$ an ideal.  The \emph{radical} $\sqrt{I}$ of $I$ is defined to be those $x\in R$ such that $x^n \in I$ for some $n\in \N$.  It is also the intersection of all prime ideals containing $I$.  Here we view the radical as a pair operation on $\cP = \{(I,R) \mid I \text{ ideal of } R\}$ (where $R$ is fixed across $\cP$), with $p(I,R) := \sqrt{I}$.
\end{plaindef}

\begin{plaindef}
Let $R$ be a Noetherian ring of positive prime characteristic $p$.  Set $R^\circ$ to be the complement of the union of the minimal primes of $R$. Let $F: R \ra R$ be the \emph{Frobenius endomorphism} $r \mapsto r^p$. For any $e\geq 0$, let $F^e_*(R)$ denote the $R$-$R$ bimodule with underlying Abelian group $R$ (with elements written $F^e_*(r)$ for each $r\in R$) with actions given by $aF^e_*(r)b := F^e_*(arb^{p^e})$ for $a,b \in R$ and $F^e_*(r) \in F^e_*(R)$.  This gives rise to the \emph{Peskine-Szpiro functors} $F^{e*}$ on the category of (left) $R$-modules, given by $F^{e*}(-) := F^e_*(R) \otimes_R -$.  For each $z\in M$, we let $z^{p^e}_M := F^e_*(1) \otimes z \in F^e_*(R) \otimes_R M=F^{e*}(M)$.  Then for any $r\in R$ and $z\in M$, we have $(rz)^{p^e}_M = r^{p^e} z^{p^e}_M$; hence the notation.  In particular, when $z\in M=R$, $z^{p^e}_M = z^{p^e}$ in the ordinary sense.  For any $R$-submodule $L \subseteq M$, let $L^{[p^e]}_M := $ the $R$-submodule of $F^{e*}(M)$ generated by $\{z^{p^e}_M \mid z\in L\}$ -- in other words, the image of $F^{e*}(i)$, where $i: L \hookrightarrow M$ is the inclusion map.  For an ideal $I$, we have $I^{[p^e]}_R = $ the ideal generated by $\{a^{p^e} \mid a\in I\}$, and $F^{e*}(R/I) \cong R/I^{[p^e]}$ naturally as $R$-modules.  When $M$ is finitely generated, then given a finite free presentation $R^t \arrow{\phi} R^s \ra M \ra 0$ for $M$, where $\phi$ is given by a matrix $[a_{ij}]_{\substack{1\leq i \leq s \\ 1 \leq j \leq t}}$ with each $a_{ij} \in R$, it follows that $F^{e*}(M) \cong \coker {\phi^{[p^e]}}$, where $\phi^{[p^e]}: R^t \ra R^s$ is given by the matrix $[a_{ij}^{p^e}]_{\substack{1\leq i \leq s \\ 1 \leq j \leq t}}$. 
%
\begin{itemize}
    \item For a pair $L \subseteq M$ of modules, the \emph{Frobenius closure $L^{\rm F}_M$ of $L$ in $M$} is the set of all $z\in M$ such that for some $e\geq 0$, $z^{p^e}_M \in L^{[p^e]}_M$ (when $M=R$, see \cite[Section 7]{HHsplit}).
    \item For a pair $L \subseteq M$ of modules, the \emph{tight closure $L^*_M$ of $L$ in $M$} \cite[Section 8]{HHmain} is the set of all $z\in M$ such that there exist $c\in R^\circ$ and $e_0 \geq 0$ such that for all $e\geq e_0$, we have $cz^{p^e}_M \in L^{[p^e]}_M$.
    \item A \emph{big weak test element} (see \cite[Section 6]{HHmain} and \cite[p. 63]{HoFNDTC}) is an element $c\in R^\circ$ such that there is some $e_0 \geq 0$ with the property that for any $R$-module inclusion $L \subseteq M$ and any $z\in L^*_M$, we have $cz^{p^e}_M \in L^{[p^e]}_M$ for all $e\geq e_0$.  We remove the word \emph{weak} when $e_0=0$.
    \item We say that $R$ is \emph{F-finite} \cite[Definition before Lemma 1.4]{Fe-Fpure} if $F^1_*(R)$ is finitely generated as a right $R$-module. 
    \item If $\ia$ is an ideal such that $\ia \cap R^\circ \neq \emptyset$, the \emph{$\ia$-tight closure $L^{*\ia}_M$ of $L$ in $M$} \cite{HaTa-gentest, HaYo-atc} is the set of all elements $z\in M$ such that there exist $c\in R^\circ$ and $e_0 \geq 0$ such that for all $e\geq e_0$, $c \ia^{p^e} z^{p^e}_M \subseteq L^{[p^e]}_M$.  Note that we are taking \emph{ordinary} (not bracket) powers of $\ia$.
    \item A \emph{big $\ia$-test element} is an element $c\in R^\circ$ such that for any $R$-modules $L \subseteq M$ and any $z \in L^{*\ia}_M$, we have $c\ia^{p^e} z^{p^e}_M \subseteq L^{[p^e]}_M$ for all $e\geq 0$.  If this is only assumed to be true when $M$ is finitely generated, we call $c$ an \emph{$\ia$-test element} \cite[Definition 1.5]{HaTa-gentest}.
    \item If $R$ is F-finite, we say that an element $z\in M$ is in the \emph{tight interior} of $M$ (written $z\in M_*$) \cite{nmeSc-tint} if for any $c\in R^\circ$ and any nonnegative integer $e_0 \geq 0$, there are integers $t\geq 1$ and $e_1, \ldots, e_t \in [e_0,+\infty) \cap \Z$ and \emph{right} $R$-linear maps (i.e. linear with respect to their right $R$-module structures) $g_i: F^{e_i}_*(R) \ra M$ such that $z = \sum_{i=1}^t g_i(F^{e_i}_*(c))$.
\end{itemize}
\end{plaindef}

\begin{plainrem}
Big weak test elements exist whenever $R$ is essentially of finite type over an excellent local ring \cite{HHbase}, or is F-finite \cite{HH-sFreg}, or is F-pure and excellent \cite{Sha-Fpure}, or is a reduced affinoid algebra \cite{DEST1}.  Thus, one does not lose much (and does gain a lot) from assuming that $R$ has such an element. Some similar statements hold for $\ia$-test elements.  Thus, for the purposes of discussing ($\ia$-)tight closure, we are happy to assume the existence of such elements.
\end{plainrem}

\begin{plaindef}[{\cite{HH-bCM,Sm-param}}]
Let $R$ be a Noetherian ring.  If $R$ is an integral domain, the \emph{absolute integral closure $R^+$ of $R$} is the integral closure of $R$ in an algebraic closure of the fraction field of $R$. Then if $L \subseteq M$ are $R$-modules, the \emph{plus closure $L^+_M$ of $L$ in $M$} is $L^{\cl_{R^+}}_M$.

If $R$ is not necessarily an integral domain, then we say $z \in L^+_M$ if for each minimal prime $\p$ of $R$, the image of $z$ in $M/\p M$ is in the plus closure of the submodule $(L+\p M) / \p M$ of $M/\p M$.
\end{plaindef}

\begin{plaindef}
If $(R,\m)$ is a local ring and $L \subseteq M$ are $R$-modules, the \emph{socle {\rm soc}$(L,M)$ of $L$ in $M$} is the submodule $(L :_M \m)$ of $M$.
\end{plaindef}

\begin{plaindef}
Let $R$ be a Noetherian ring, $\ia$ an ideal, and $L \subseteq M$ be $R$-modules.  Then the \emph{$\ia$-saturation of $L$ in $M$}, written $(L :_M \ia^\infty)$, is $\{x \in M \mid \ia^t x \subseteq L \text{ for some } t\in \N\}$. This operation is typically unnamed in the literature but is a natural generalization of the definition of the $\ia$-saturation of an ideal (see e.g. \cite[17.4.5]{BrSh-locobook2}). It is at the base of the theory of local cohomology.
\end{plaindef}

\begin{plaindef}[
{\cite[Definition 1.1]{nmeUlr-lint}; derived from \cite{Rees-redmod}, \cite{EHU-Ralg} and others}]
Let $L \subseteq F$
be $R$-modules such that $F$ is free. Let $S=\Sym(F)$ be the symmetric algebra over $R$ defined by $F$, with its natural grading over $R$. Let $T$ be the $R$-subalgebra of $S$ generated by $L$, considered as a submodule of the degree 1 component $F$ of $S$. The \emph{integral closure $L_F^{-}$ of $L$ in $F$} is the degree 1 part of the (ring-theoretic) integral closure of $T$ in $S$. 
\end{plaindef}

\begin{plaindef}[{\cite[Definition 1.3]{nmeUlr-lint}, derived from \cite{EHU-Ralg}}]
Let $R$ be a Noetherian ring and $L \subseteq M$ be finite $R$-modules.  For $x \in M$, we say $x$ is in the $EHU$-integral closure of $L$ in $M$,  written $x \in L_M^{\EHUi}$ if for every $R$-module homomorphism $g:M \rightarrow F$ where $F$ is free we have $g(x) \in g(L)_F^{-}$.
\end{plaindef}

\begin{plaindef}[\cite{nmeUlr-lint}] \label{def:li}
Let $L \subseteq M$ be $R$-modules. Let $\pi:F \twoheadrightarrow M$ be a surjection of a free $R$-module $F$ onto $M$. Let $K:=\pi^{-1}(L)$. Then the liftable integral closure of $L$ in $M$ is $\lic LM:=\pi(K_F^{-})$.
\end{plaindef}

\begin{plaindef}[derived from \cite{Rees-redmod}]
Let $R$ be a Noetherian ring and $L \subseteq M$ be $R$-modules.  We say $x\in M$ is in the \emph{Rees integral closure $L^\Ri_M$ of $L$ in $M$} if for every minimal prime $\p$ of $R$ and every discrete rank 1 valuation ring $V$ between $R/\p$ and its fraction field $\kappa(\p)$, the image of $x$ in the map $M \ra M \otimes_R \kappa(\p)$ (induced by the natural ring map $R \ra \kappa(\p)$) is in the image of the composite map $L \otimes_R V \ra M \otimes_R V \ra M \otimes_R \kappa(\p)$ (induced by the inclusions $L \hookrightarrow M$ and $V \hookrightarrow \kappa(\p)$).
\end{plaindef}

\begin{plaindef}
Let $(R,\m)$ be a Noetherian local ring and let $L \subseteq M$ be a nested pair of $R$-modules such that $M/L$ has finite length.  The \emph{basically full closure of $L$ in $M$} is defined to be $(\m L :_M \m)$ \cite[Definition 4.4]{HRR-bf}.  That is, $\Jcol \m L M$.
\end{plaindef}

\begin{plaindef}[{see \cite{RatRu-rr} for $I=R$; \cite[Example 2.6]{ERGV-nonres} in general}]
Let $J \subseteq I$ be ideals.  The \emph{Ratliff-Rush operation of $J$ in $I$} is $\bigcup_{n \in \N_0} (J^{n+1} :_I J^n)$.  
\end{plaindef}

\providecommand{\bysame}{\leavevmode\hbox to3em{\hrulefill}\thinspace}
\providecommand{\MR}{\relax\ifhmode\unskip\space\fi MR }
\providecommand{\MRhref}[2]{%
  \href{http://www.ams.org/mathscinet-getitem?mr=#1}{#2}
}
\providecommand{\href}[2]{#2}

\end{document}